\documentclass[12pt,reqno, final]{amsart}	
	
		\title{$L^2$-Betti numbers of coamenable quantum groups}
		\author{David Kyed}
		\address{David Kyed,
		Mathematisches Institut,
		Georg-Au\-gust-Uni\-versi\-t{\"a}t, G{\"o}t\-ting\-en,
		Bunsenstra{\ss}e 3-5,
		D-37073 G{\"o}ttingen, 
		Germany}
	
		\email{kyed@uni-math.gwdg.de}
		\urladdr{www.uni-math.gwdg.de/kyed}
\usepackage[english]{babel}              % Dansk orddeling osv.
\usepackage[latin1]{inputenc}            % ï¿½ï¿½ï¿½ï¿½ï¿½ï¿½
\usepackage{amsmath}                     % Matematiske kommandoer
\usepackage{amssymb}                     % Matematiske symboler
\usepackage{amsthm}
\usepackage{fancybox}
\usepackage{graphicx}
\usepackage{hyperref}
\usepackage{mathrsfs}									  % Ekstra krï¿½llede bogstaver
\usepackage[all]{xy}										% Lï¿½kre kommutative diagrammer
\usepackage[T1]{fontenc}								
\usepackage{url}												% %Indsï¿½ttelse af www-links
\usepackage{array}
\usepackage{verbatim,epsfig}
\usepackage{wasysym}                    % makes it possible to use \ocircle
\theoremstyle{plain}

\newtheorem{thm}{Theorem}[section] % Her kunne vi skrive subsection i stedet
\newtheorem{cor}[thm]{Corollary}
\newtheorem{lem}[thm]{Lemma}
\newtheorem{defi}[thm]{Definition}
\newtheorem{prop}[thm]{Proposition}

\theoremstyle{definition}

\newtheorem{ex}[thm]{Example}
\newtheorem{rem}[thm]{Remark}

\newtheorem{obs}[thm]{Observation}

\newcommand{\NN}{{\mathbb N}}
\newcommand{\ZZ}{{\mathbb Z}}

\newcommand{\RR}{{\mathbb R}}
\newcommand{\CC}{{\mathbb C}}

\newcommand{\MM}{{\mathbb M}}
\newcommand{\GG}{{\mathbb G}}

\newcommand{\ip}[2]{\langle {#1} \hspace{0.03cm} | \hspace{0.03cm} {#2} \rangle}

\newcommand{\ipp}{\ip{\cdot}{\cdot}}

\renewcommand{\L}{{\mathscr L}}

\newcommand{\K}{{\mathcal K}}
\newcommand{\B}{{\mathscr B}}

\newcommand{\I}{{\mathbf{I}}}

\newcommand{\varps}{{\varepsilon}}
\newcommand{\rg}{{\operatorname{rg\hspace{0.04cm}}}}
\newcommand{\htens}{\bar{\otimes}}

\newcommand{\tens}{\otimes}

\renewcommand{\Re}{{\operatorname{Re}}}

\newcommand{\spann}{{\operatorname{span}}}

\newcommand{\To}{\longrightarrow}

\newcommand{\supp}{{\operatorname{supp}}}
\newcommand{\red}{{\operatorname{red}}}
\newcommand{\Tor}{\operatorname{Tor}}
\newcommand{\op}{{\operatorname{{op}}}}

\newcommand{\Mor}{\operatorname{Mor}}

\newcommand{\id}{\operatorname{id}}

\newcommand{\tr}{{\operatorname{Tr}}}

\newcommand{\bet}{\beta^{(2)}}

\newcommand{\del}{{\partial}}

\newcommand{\alg}{{\operatorname{alg}}}

\newcommand{\tetxrm}{\textrm}

\newcommand{\utens}[1]{\underset{#1}{\tens}}
\newcommand{\Irred}{\operatorname{Irred}}
\newcommand{\vna}{\operatorname{vNa}}
\newcommand{\hooklongrightarrow}{\lhook\joinrel\longrightarrow}

\newcommand{\NW}{{\operatorname{NW}}}
\newcommand{\FC}{{\operatorname{FC}}}
\newcommand{\A}{{\operatorname{A}}}
\newcommand{\tenrep}{\mbox{ $\mbox{\scriptsize \sf T}
\hspace{-1.77ex}\bigcirc$}} %70

\newcommand{\pfb}{P_{\bar{F}}}
\newcommand{\pdb}{P_{\bar{\del}}}
\newcommand{\full}{\operatorname{full}}
\renewcommand{\B}{B}
\newcommand{\tensrep}{\tenrep}
\renewcommand{\S}{{\mathscr S}}
\newcommand{\Tr}{\operatorname{Tr}}
\newcommand{\minitenrep}{\mbox{ $\mbox{\tiny \sf T}
\hspace{-1.44ex} \ocircle$}}

\begin{document}
\begin{abstract}
We prove that a compact quantum group is coamenable if and only if its corepresentation ring is amenable. We further propose a F{\o}lner condition for compact quantum groups and prove it to be equivalent to coamenability. Using this F{\o}lner condition, we prove that for a coamenable compact quantum group with tracial Haar state, the enveloping von Neumann algebra is dimension flat over the Hopf algebra of matrix coefficients. This generalizes a theorem of L{\"u}ck from the group case to the quantum group case, and provides examples of compact quantum groups with vanishing $L^2$-Betti numbers.
\end{abstract}
\maketitle
\section*{Introduction}\label{intro-section}
The theory of $L^2$-Betti numbers for discrete groups is originally due to Atiyah and dates back to the seventies \cite{atiyah}. These $L^2$-Betti numbers are defined for those discrete groups that permit a free, proper and cocompact action on some contractible, Riemannian manifold $X$. If $\Gamma$ is such a group, the space of square integrable $p$-forms on $X$ becomes a finitely generated Hilbert module for the group von Neumann algebra $\L(\Gamma)$. As such it has a Murray-von Neumann dimension which turns out to be independent of the choice of $X$ and is called the $p$-th $L^2$-Betti number of $\Gamma$, denoted $\bet_p(\Gamma)$. More recently, L{\"u}ck \cite{luck97,luck98,luckto} transported the notion of Murray-von Neumann dimension to the setting of finitely generated projective  (algebraic) $\L(\Gamma)$-modules and extended thereafter the domain of definition to the class of all modules.  With this extended dimension function, $\dim_{\L(\Gamma)}(-)$, it is possible to extend the notion of $L^2$-Betti numbers to cover all discrete groups $\Gamma$ by setting
\[
\bet_p(\Gamma)=\dim_{\L(\Gamma)}\Tor_{p}^{\CC\Gamma}(\L(\Gamma),\CC).
\]
For more details on the relations between the different definitions of $L^2$-Betti numbers and the extended dimension function we refer to L{\"u}ck's book \cite{luck02}.\\
All the ingredients in the homological algebraic definition above have fully developed analogues in the world of compact quantum groups, and using this dictionary  the notion of $L^2$-Betti numbers was generalized to the quantum group setting in  \cite{quantum-betti}. Since this generalization is central for the work in the present paper, we shall now explain it in greater detail.  Consider a compact quantum group $\GG=(A,\Delta)$ and assume that its Haar state $h$ is a trace. If we denote by $A_0$ the unique dense Hopf $*$-algebra and by $M$ the enveloping von Neumann algebra of $A$ in the GNS representation arising  from $h$, then the $p$-th $L^2$-Betti number of $\GG$ is defined as
\[
\bet_p(\GG)=\dim_M\Tor_p^{A_0}(M,\CC).
\]
Here $\CC$ is considered an $A_0$-module via the counit $\varps\colon A_0\to\CC$ and $\dim_M(-)$ is L{\"u}ck's extended dimension function arising from (the extension of) the trace-state $h$. This definition extends the classical one \cite[1.3]{quantum-betti} in the sense that 
\[
\bet_p(\GG)=\bet_p(\Gamma)
\]
when $\GG=(C^*_\red(\Gamma),\Delta_\red)$. \\
The aim of this paper is to investigate the $L^2$-Betti numbers of the class of coamenable, compact quantum groups. In the classical case we have that $\bet_p(\Gamma)=0$ for all $p\geq 1$ whenever $\Gamma$ is an amenable group. This can be seen as a special case of \cite[5.1]{luck98} where it is proved that the von Neumann algebra $\L(\Gamma)$ is \emph{dimension flat} over $\CC\Gamma$, meaning that
\begin{align*}
\dim_{\L(\Gamma)}\Tor_p^{\CC\Gamma}(\L(\Gamma),Z)=0 \tag{$p\geq 1$}
\end{align*}
for any $\CC\Gamma$-module $Z$ --- provided, of course, that $\Gamma$ is still assumed amenable. We generalize this result to the quantum group setting in Theorem \ref{dim-flad}. More precisely, we prove that if $\GG=(A,\Delta)$ is a compact, coamenable quantum group with tracial Haar state and $Z$ is any module for the algebra of matrix coefficients $A_0$ then
\begin{align*}
\dim_M\Tor_p^{A_0}(M,Z)=0. \tag{$p\geq 1$}
\end{align*}
Here $M$ is again the enveloping von Neumann algebra in the GNS representation arising from the Haar state. In order to prove this result we need a \emph{F{\o}lner condition} for compact quantum groups. The classical F{\o}lner condition for groups \cite{foelner} is a geometrical condition, on the action of the group on itself, which is equivalent to amenability of the group. In order to obtain a quantum analogue of F{\o}lner's condition a detailed study of the  ring of corepresentations, associated to a compact quantum group, is needed. The ring of corepresentations is a special case of a so-called fusion algebra and we have therefore devoted a substantial part of this paper to the study of abstract fusion algebras and their amenability. Amenability for (finitely generated) fusion algebras was introduced by Hiai and Izumi in \cite{izumi} where they also gave two equivalent F{\o}lner-type conditions for fusion algebras. We generalize their results to the non-finitely generated case and prove that a compact quantum group is coamenable if and only if its corepresentation ring is amenable. From this we obtain a F{\o}lner condition for compact quantum groups which is equivalent to coamenability. Using this F{\o}lner condition we prove our main result, Theorem \ref{dim-flad}, which implies that coamenable compact quantum groups have vanishing $L^2$-Betti numbers in all positive degrees.

\vspace{0.3cm}
\paragraph{\emph{Structure.}}
The paper is organized as follows. In the first section we recapitulate (parts of) Woronowicz's theory of compact quantum groups. The second and third section is devoted to the study of abstract fusion algebras and amenability of such. In the fourth section we discuss coamenability of compact quantum groups and investigate the relation between coamenability of a compact quantum group and amenability of its corepresentation ring. The fifth section is an interlude in which the necessary notation concerning von Neumann algebraic compact quantum groups and their discrete duals is introduced. The sixth section is devoted to the proof of our main theorem (\ref{dim-flad}) and the seventh, and final, section consists of examples.
\vspace{0.3cm}
\paragraph{\emph{Acknowledgements.}} I wish to thank my supervisor Ryszard Nest for the many discussions about quantum groups and their (co)amenability, and Andreas Thom for pointing out to me that the bicrossed product construction could be used to generate examples of quantum groups satisfying F{\o}lner's condition.
\vspace{0.3cm}
\paragraph{\emph{Notation.}} Throughout the paper, the symbol $\odot$ will be used to denote algebraic tensor products while the symbol $\htens$ will be used to denote tensor products in the category of Hilbert spaces or the category of von Neumann algebras. All tensor products between $C^*$-algebras are assumed minimal/spatial and these will be denoted by the symbol $\tens$.

\section{Preliminaries on compact quantum groups}\label{prelim}
In this section we briefly recall Woronowicz's theory of compact quantum groups. Detailed treatments, and proofs of the results stated, can be found in \cite{wor-cp-qgrps}, \cite{vandaele} and \cite{tuset}. \\
A compact quantum group $\GG$ is a pair $(A,\Delta)$ where $A$ is a unital $C^*$-algebra and $\Delta\colon A\To A\tens A$ is a unital $*$-homomorphism from $A$ to the minimal tensor product of $A$ with itself satisfying:
\begin{align*}
(\id\tens \Delta)\Delta &=(\Delta\tens \id)\Delta \tag{coassociativity}\\
\overline{\Delta(A)(1\tens A)}&=\overline{\Delta(A)(A\tens 1)}= A\tens A &\tag{non-degeneracy}
\end{align*}
For such a compact quantum group $\GG=(A,\Delta)$, there exists a unique state $h\colon A\to \CC$, called the Haar state, which is invariant in the sense that
\[
(h\tens \id)\Delta(a)=(\id\tens h)\Delta(a)=h(a)1,
\]
for all $a\in A$. Let $H$ be a Hilbert space and let $u\in M(\K(H)\tens A)$ be an invertible multiplier. Then $u$ is called a \emph{corepresentation} if
\[
(\id\tens \Delta)u=u_{(12)}u_{(13)},
\]
where we use the standard \emph{leg numbering convention}; for instance $u_{(12)}=u\tens 1$. \emph{Intertwiners}, \emph{direct sums} and \emph{equivalences} between corepresentations as well as \emph{irreducibility} are defined in a straight forward manner. See e.g.~\cite{vandaele} for details. We shall denote by $\Mor(u,v)$ the set of intertwiners from $u$ to $v$. It is a fact that each irreducible corepresentation is finite dimensional and equivalent to a unitary corepresentation. Moreover, every unitary corepresentation is unitarily equivalent to a direct sum of irreducible corepresentations. For two finite dimensional unitary corepresentations $u,v$ their \emph{tensor product} is defined as
\[
u\tenrep v=u_{(13)}v_{(23)}.
\]
This is again a unitary corepresentation of $\GG$. The algebra $A_0$ generated by all matrix coefficients arising from irreducible corepresentations becomes a Hopf $*$-algebra (with the restricted comultiplication) which is dense in $A$. We denote its antipode by $S$ and its counit by $\varps$. We also recall that the restriction of the Haar state to the $*$-algebra $A_0$ is always faithful. The quantum group $\GG$ is called a compact \emph{matrix} quantum group if there exists a \emph{fundamental} unitary corepresentation; i.e.~a finite dimensional, unitary corepresentation whose matrix coefficients generate $A_0$ as a $*$-algebra. Each finite dimensional, unitary corepresentation $u$ defines a  \emph{contragredient} corepresentation $u^c$ on the dual Hilbert space; if $u\in \B(H)\odot A_0$ for some finite dimensional Hilbert space $H$ then $u^c\in B(H')\odot A_0$ is given by $u^c=((\ \cdot \ )'\tens S)u$, where for $T\in \B(H)$ the operator $T'\in B(H')$ is the natural dual $(T'(y'))(x)=y'(Tx)$. In general $u^c$ is not a unitary, but it is a corepresentation; i.e.~it is invertible and satisfies $(\id\tens \Delta)u^c=u^c_{(12)}u^c_{(13)}$ and is therefore equivalent to a unitary corepresentation.  By choosing an orthonormal basis $e_1,\dots, e_n$ for $H$ we get an identification of $\B(H)\odot A_0$ with $\MM_n(A_0)$. If, under this identification, $u$ becomes the matrix $(u_{ij})$ then $u^c$ is identified with the matrix $\bar{u}=(u_{ij}^*)$, where we identify $B(H')\odot A_0$ with $\MM_n(A_0)$ using the dual basis $e_1',\dots, e_n'$. From this it follows that $u^{cc}$ is equivalent to $u$. Note also that one has $(u\oplus v)^c=u^c\oplus v^c$ and $(u\tenrep v)^c=v^c\tenrep u^c$ for unitary corepresentations $u$ and $v$ (see e.g.~\cite{woronowicz-pseudo}). If $u\in \B(H)\odot A_0$ is a finite dimensional corepresentation its \emph{character} is defined as
\[
\chi(u)=(\tr\tens\id)u\in A_0,
\]
where $\tr$ is the unnormalized trace on $\B(H)$. The character map has the following properties.
\begin{prop}[\cite{woronowicz-pseudo}]\label{character}
If $u$ and $v$ are finite dimensional, unitary corepresentations then
\[
\chi(u\tenrep v)=\chi(u)\chi(v),  \quad \chi(u\oplus v)=\chi(u)+\chi(v) \quad \textrm{ and } \quad \chi(u^c)=\chi(u)^* .
\]
Moreover, if  $u$ and $v$ are equivalent then $\chi(u)=\chi(v)$. 
\end{prop}

We end this section with the two basic examples of compact quantum groups arising from actual groups.
\begin{ex}
If $G$ is a compact, Hausdorff topological group then the Gelfand dual  $C(G)$ becomes a compact quantum group with comultiplication $\Delta_c\colon C(G)\To C(G)\tens C(G)=C(G\times G)$ given by
\[
\Delta_c(f)(s,t)=f(st).
\]
The Haar state is in this case given by integration against the Haar probability measure on $G$, and the finite dimensional unitary corepresentations of $C(G)$ are exactly the finite dimensional unitary representations of $G$. 

\end{ex}
\begin{ex}	
If $\Gamma$ is a discrete, countable group then the reduced group $C^*$-algebra $C^*_\red(\Gamma)$ becomes a compact quantum group when endowed with comultiplication given by
\[
\Delta_\red(\lambda_\gamma)=\lambda_\gamma\tens\lambda_\gamma.
\]
Here $\lambda$ denotes the left regular representation of $\Gamma$. In this case, the Haar state is just the natural trace on $C^*_\red(\Gamma)$, and a complete family of irreducible, unitary corepresentations is given by the set $\{\lambda_\gamma\mid \gamma\in \Gamma\}$.
\end{ex}
\begin{rem}\label{sep-rem}
All compact quantum groups to be considered in the following are assumed to have a separable underlying $C^*$-algebra.
The quantum Peter-Weyl theorem \cite[3.2.3]{tuset} then implies that the GNS space arising from the Haar state is separable and, in particular, that there are at most countable many (pairwise inequivalent) irreducible corepresentations.
\end{rem}

\section{Fusion Algebras}\label{fusion-section}
In this section we introduce the notion of fusion algebras and amenability of such objects. This topic was treated by Hiai and Izumi  in \cite{izumi} and we will follow this reference closely throughout this section. Other references on the subject are \cite{yamagami}, \cite{yamagami-hayashi} and \cite{sunder-bimodules}. Throughout the section, $\NN_0$ will denote the non-negative integers.

\begin{defi}[\cite{izumi}]\label{fusion-defi}
Let $R$ be a unital ring and  assume that $R$ is free as $\ZZ$-module with basis $I$. Then $R$ is called a fusion algebra if the unit $e$ is an element of $I$ and the following holds:
\begin{itemize}
\item[(i)] The abelian monoid $\NN_0[I]$ is stable under multiplication. That is, for all $\xi, \eta \in I$ the unique family $(N_{\xi,\eta}^\alpha)_{\alpha\in I}$  of integers satisfying
\[
\xi\eta=\sum_{\alpha\in I} N_{\xi,\eta}^\alpha \alpha,
\]
consists of non-negative numbers. 
\item[(ii)] The ring $R$ has a $\ZZ$-linear, anti-multiplicative involution $x\mapsto \bar{x}$ preserving the basis $I$ globally.
\item[(iii)] Frobenius reciprocity holds, i.e.~for $\xi,\eta,\alpha\in I$ we have
\[
N_{\xi,\eta}^\alpha=N_{\bar{\xi},\alpha}^\eta=N_{\alpha,\bar{\eta}}^\xi.
\]
\item[(iv)] There exists a $\ZZ$-linear multiplicative function $d\colon R \to [1,\infty[$ such that $d(\xi)=d(\bar{\xi})$ for all $\xi\in I$. This function is called the dimension function.
\end{itemize} 
\end{defi}
Note that the distinguished basis, involution and dimension function are all included in the data defining a fusion algebra. Each fusion algebra comes with a natural trace $\tau$ given by 
\[
\sum_{\alpha\in I } k_\alpha\alpha\overset{\tau}{\longmapsto} k_e. 
\]
We shall use this trace later to define a $C^*$-envelope of a fusion algebra. Note also that the multiplicativity of $d$ implies
\[
1=\sum_{\alpha\in I}\frac{d(\alpha)}{d(\xi)d(\eta)}N_{\xi,\eta}^\alpha,
\]
for all $\xi,\eta\in I$. For an element $r=\sum_{\alpha\in I} k_\alpha \alpha \in R$, the set $\{\alpha\in I \mid k_\alpha\neq 0\}$ is called the support of $r$ and denoted $\supp(r)$. We shall also consider the complexified fusion algebra $\CC\tens_{\ZZ} \ZZ[I]$ which will be denoted $\CC[I]$ in the following. Note that this becomes a complex $*$-algebra with the induced algebraic structures.

\begin{ex}
For any discrete group $\Gamma$ the integral group ring $\ZZ [\Gamma]$ becomes a fusion algebra when endowed with (the $\ZZ$-linear extension of) inversion as involution and trivial dimension function given by $d(\gamma)=1$ for all $\gamma\in \Gamma$.
\end{ex}
The irreducible representations of a compact group constitute the basis in a fusion algebra where the tensor product of representations is the product.  We shall not go into details with this construction since it will be contained in the following more general example.
\begin{ex}\label{corep-fusion}
If $\GG=(A,\Delta)$ is a compact quantum group its irreducible corepresentations constitute the basis of a fusion algebra with tensor product as multiplication. Since this example will play a prominent role later, we shall now elaborate on the construction. Denote by $\Irred(\GG)=(u^\alpha)_{\alpha\in I}$ a complete family of representatives for the equivalence classes of irreducible, unitary corepresentations of $\GG$. As explained in Section \ref{prelim},  for all $u^\alpha,u^\beta\in \Irred(\GG)$ there exists a finite subset $I_0\subseteq I$ and a family $(N_{\alpha,\beta}^{\gamma})_{\gamma\in I_0}$ of positive integers  such that  $u^\alpha\tensrep u^\beta$ is equivalent to
\[
\bigoplus_{\gamma\in I_0} \underbrace{u^\gamma\oplus\cdots\oplus u^\gamma}_{N_{\alpha\beta}^\gamma \text{ times}}.
\]
Thus, a product can be defined on the free $\ZZ$-module $\ZZ[\Irred(\GG)]$ by setting
\[
u^\alpha \cdot u^\beta=\sum_{\gamma\in I_0}N_{\alpha,\beta}^\gamma u^\gamma,
\]
and the trivial corepresentation $e=1_A\in \Irred(\GG)$ is  a unit for this product. If we denote by $u^{\bar{\alpha}}\in \Irred(\GG)$ the unique representative equivalent to $(u^\alpha)^c$, then the map $u^\alpha\mapsto u^{\bar{\alpha}}$ extends to a conjugation on the ring $\ZZ[\Irred(\GG)]$ and since each $u^\alpha$ is an element of $\MM_{n_{\alpha}}(A)$ for some $n_\alpha\in \NN$ we can also define a dimension function $d\colon \ZZ[\Irred(\GG)]\to[1,\infty[$ by $d(u^{\alpha})=n_\alpha$. When endowed with this multiplication, conjugation and dimension function $\ZZ[\Irred(\GG)]$ becomes a fusion algebra. The only thing that is not clear at this moment is that Frobenius reciprocity holds. To see this, we first note that for any $\alpha\in I$ and any finite dimensional corepresentation $v$ we have (by Schur's Lemma \cite[6.6]{vandaele}) that $u^\alpha$ occurs exactly
\[
\dim_\CC\Mor(u^\alpha,v)
\]
times in the decomposition of $v$. Moreover, we have for any two unitary core\-presentations $v$ and $w$ that
\begin{align*}
\dim_\CC\Mor(v,w)&=\dim_\CC((V_w\tens V_v')^{w\minitenrep v^c})\\
\dim_\CC\Mor(v^{cc},w)&=\dim_\CC((V_v'\tens V_w)^{v^c\minitenrep w})
\end{align*}
Here the right hand side denotes the linear dimension of the space of invariant vectors under the relevant coaction. These formulas are proved in \cite[3.4]{woronowicz-pseudo} for compact matrix quantum groups, but the same proof carries over to the case where the compact quantum group in question does not necessarily possess a fundamental corepresentation. Using the first formula, we get for $\alpha,\beta,\gamma\in I$ that
\begin{align*}
N_{\alpha,\beta}^\gamma &=\dim_\CC\Mor(u^\gamma,u^\alpha\tenrep u^\beta)\\
&=\dim_\CC(V_\alpha\tens V_\beta\tens V_\gamma')^{u^\alpha\minitenrep u^\beta\minitenrep (u^\gamma)^c}\\
&=\dim_\CC(V_\gamma\tens V_\beta '\tens V_\alpha ')^{u^\gamma\minitenrep (u^\beta)^c\minitenrep (u^\alpha)^c}\\
&=\dim_\CC \Mor(u^\alpha,u^\gamma\tenrep (u^\beta)^c)\\
&=N_{\gamma,\bar{\beta}}^{\alpha}.
\end{align*}
%For at indse at man den oprindelige corep's invariante vektorer har samme dimension som den for den kontragradiente bruges feks Woronowich ''compact matrix pseudo grps'' prop 4.3. Detaljer er at finde pï¿½ et hï¿½ndskrevet ark
The remaining identity in Frobenius reciprocity follows similarly using the second formula. The fusion algebra $\ZZ[\Irred(\GG)]$ is called the corepresentation ring (or fusion ring) of  $\GG$ and is denoted $R(\GG)$. 

Recall that the character of a corepresentation $u\in \MM_n(A)$ is defined as $\chi(u)=\sum_{i=1}^n u_{ii}$. It follows from Proposition \ref{character} that the $\ZZ$-linear extension
\[
\chi\colon \ZZ[\Irred(\GG)]\To A_0
\]
is an injective homomorphism of $*$-rings. I.e.~$\chi$ is additive and multiplicative with $\chi(u^{\bar{\alpha}})=(\chi(u^\alpha))^*$. This gives a link between the two $*$-algebras $R(\GG)$ and $A_0$ which will be of importance later.

\end{ex}
Other interesting examples of fusion algebras arise from inclusions of $\I\I_1$-factors. See \cite{izumi} for details. 
\begin{rem}
In the following we shall only consider fusion algebras with an at most countable basis. This will therefore be assumed without further notice throughout the paper. Since we will primarily be interested in corepresentation rings of compact quantum groups, this is not very restrictive since the standing separability assumption (Remark \ref{sep-rem}) ensures that the corepresentation rings always have a countable basis.
\end{rem}

Consider again an abstract fusion algebra $R=\ZZ[I]$. For $\xi,\eta\in I$ we define the (weighted) convolution of the corresponding Dirac measures, $\delta_\xi$ and $\delta_\eta$, as
\[
\delta_\xi\ast\delta_\eta=\sum_{\alpha\in I}\frac{d(\alpha)}{d(\xi)d(\eta)}N_{\xi,\eta}^\alpha \delta_\alpha \in \ell^1(I).
\]
This extends linearly and continuously to a submultiplicative product on $\ell^1(I)$. For  $f\in \ell^\infty(I)$ and $\xi\in I$ we define $\lambda_\xi(f),\rho_\xi(f)\colon I\to \CC$ by
\begin{align*}
\lambda_\xi(f)(\eta)&=\sum_{\alpha\in I} f(\alpha)(\delta_{\bar{\xi}}\ast\delta_\eta)(\alpha) \\
\rho_\xi(f)(\eta)&=\sum_{\alpha\in I} f(\alpha)(\delta_{\eta}\ast\delta_\xi)(\alpha)
\end{align*}
Denote by $\sigma$ the counting measure on $I$ scaled with $d^2$; that is $\sigma(\xi)=d(\xi)^2$. Combining Proposition 1.3, Remark 1.4 and Theorem 1.5 in \cite{izumi} we get
\begin{prop}[\cite{izumi}]\label{gns-and-lambda}
For each $f\in \ell^\infty(I)$ we have $\lambda_\xi(f),\rho_\xi(f)\in \ell^\infty(I)$ and for each $p\in \NN\cup\{\infty\}$ the maps $\lambda_\xi,\rho_\xi\colon \ell^\infty(I)\to\ell^\infty(I)$ restrict to bounded operators on $\ell^p(I,\sigma)$ denoted $\lambda_{p,\xi}$ and $\rho_{p,\xi}$ respectively. By linear extension, we therefore obtain a map $\lambda_{p,-}\colon\ZZ[I]\to B(\ell^p(I,\sigma))$ and this map respects the weighted convolution product. Moreover, for $p=2$ the operator $U\colon \ell^2(I)\to \ell^2(I,\sigma)$ given by $U(\delta_\eta)=\frac{1}{d(\eta)}\delta_\eta$ is unitary and intertwines $\lambda_{2,\xi}$ with the operator
\[
l_\xi : \delta_\eta\longmapsto \frac{1}{d(\xi)} \sum_\alpha N_{\xi,\eta}^\alpha \delta_\alpha.
\] 
\end{prop}
\begin{rem}\label{bounded-rem}
Under the natural identification of $\ell^2(I)$ with the GNS space $L^2(\CC[I],\tau)$, we see that $\pi_\tau(\xi)=d(\xi)l_\xi$. In particular the GNS representation consists of bounded operators. Here $\tau$ is the natural trace defined just after Definition \ref{fusion-defi}.
\end{rem}

\section{Amenability for Fusion Algebras}

The notion of amenability for fusion algebras was introduced in \cite{izumi}, but only in the slightly restricted setting
of finitely generated fusion algebras; a fusion algebra $R=\ZZ[I]$ is called \emph{finitely generated} if there exists a finitely supported probability measure $\mu$ on $I$ such that 
\[
I=\bigcup_{n\in \NN} \supp(\mu^{\ast n}) \quad \textrm{ and } \quad  \mu(\bar{\xi})=\mu(\xi)
\text{ for all }\xi\in I.
\]
That is, if the union of the supports of all powers of $\mu$, with respect to convolution, is $I$ and $\mu$ is invariant under the involution. The first condition is referred to as \emph{non-degeneracy} of $\mu$ and the second condition is referred to as \emph{symmetry} of $\mu$.  

In \cite{izumi}, amenability is defined, for a finitely generated fusion algebra, by requiring that $\|\lambda_{p,\mu}\|=1$ for some $1<p<\infty$ and some finitely supported, symmetric, non-degenerate probability measure $\mu$. It is then proved that this is independent of the choice of $\mu$ and $p$, using the non-degeneracy property of the measure. If we consider a compact quantum group $\GG=(A,\Delta)$ it is not difficult to prove that its corepresentation ring $R(\GG)$ is finitely generated exactly when $\GG$ is a compact matrix quantum group. Since we are also interested in quantum groups without  a fundamental corepresentation we will choose the following definition of amenability.

\begin{defi}
A fusion algebra $R=\ZZ[I]$ is called amenable if $1\in \sigma(\lambda_{2,\mu})$ for every finitely supported, symmetric probability measure $\mu$ on $I$.
\end{defi}
Here $\sigma(\lambda_{2,\mu})$ denotes the spectrum of the operator $\lambda_{2,\mu}$. From Proposition 1.3 and Corollary 4.4 in \cite{izumi} it follows that our definition agrees with the one in \cite{izumi} on the class of  finitely generated fusion algebras. The relation between amenability for fusion algebras and the classical notion of amenability for groups will be explained later. See e.g.~Remark \ref{foelner-rem} and Corollary \ref{group-fusion}.

\begin{defi}\label{rand-defi}
Let $R=\ZZ[I]$ be a fusion algebra. For two finite subsets $S,F\subseteq I$ we define the boundary of $F$ relative to $S$ as the set
\begin{align*}
\del_S(F)&=\{\alpha\in F\mid \exists \ \xi\in S: \supp(\alpha\xi)\nsubseteq F\}\\
        &\cup \{\alpha\in F^c\mid \exists \ \xi\in S: \supp(\alpha\xi)\nsubseteq F^c\}.
\end{align*}
Here, and in what follows, $F^c$ denotes the set $I\setminus F$.
\end{defi}

The modified definition of amenability allows the following extension of \cite[4.6]{izumi} from where we also adopt some notation.

\begin{thm}\label{izumi-thm}
Let $R=\ZZ[I]$ be a fusion algebra with dimension function $d$. Then the following are equivalent:
\begin{itemize}
\item[(A)] The fusion algebra is amenable. 
\item[$(\FC 1)$] For every finitely supported, symmetric probability measure $\mu$ on $I$ with $e\in \supp(\mu)$ and every $\varps>0$ there exists a finite  subset $F\subseteq I$ such that
\begin{align*}
\sum_{\xi\in \supp(\chi_F\ast \mu)}d(\xi)^2<(1+\varps)\sum_{\xi\in F}d(\xi)^2.
\end{align*}  
\item[$(\FC 2)$] For every finite, non-empty subset $S\subseteq I$ and every $\varps> 0$ there exists a finite subset $F\subseteq I$ such that
\begin{align*}
\forall \ \xi\in S : \ \|\rho_{1,\xi}(\chi_F)-\chi_F\|_{1,\sigma}<\varps \|\chi_F\|_{1,\sigma},
\end{align*}
where $\rho_{1,\xi}\in B(\ell^1(I,\sigma))$ is the operator from Proposition \ref{gns-and-lambda}.
\item[$(\FC 3)$]  For every finite, non-empty subset $S\subseteq I$ and every $\varps>0$ there exists a finite subset $F\subseteq I$ such that
\begin{align*}
\sum_{\xi\in \del_S(F)}d(\xi)^2<\varps\sum_{\xi\in F}d(\xi)^2.
\end{align*}
\end{itemize}
\end{thm}

The condition ($\FC 3$) was not present in \cite{izumi}. It is to be considered as a fusion algebra analogue of the F{\o}lner condition for groups as it is  presented in \mbox{\cite[F.6]{benedetti}}. The strategy for the proof of Theorem \ref{izumi-thm} is to prove the following implications:
\[
(\A)\Leftrightarrow (\FC 2) \Rightarrow (\FC 3) \Rightarrow (\FC 1)\Rightarrow (\FC 2).
\]
The proof of the implications $(\A)\Leftrightarrow (\FC 2)$ and $(\FC 1)\Rightarrow (\FC 2)$ are small modifications of the corresponding proof in \cite{izumi}.
We first set out to prove the circle of implications
\[
(\FC2)\Rightarrow (\FC 3)\Rightarrow (\FC 1)\Rightarrow (\FC2).
\]
For the proof we will need the following simple lemma.
\begin{lem}\label{vurd-lem}
If $N_{\xi,\eta}^\alpha>0$ for some $\xi,\eta,\alpha\in I$ then  $d(\alpha)d(\eta)\geq d(\xi)$.
\end{lem}
\begin{proof}
By Frobenius reciprocity, we have $N_{\xi,\eta}^\alpha= N_{\alpha,\bar{\eta}}^\xi>0$ and hence
\[
d(\alpha)d(\eta)=d(\alpha)d(\bar{\eta})=\sum_{\gamma}N_{\alpha,\bar{\eta}}^\gamma d(\gamma)\geq N_{\alpha,\bar{\eta}}^\xi d(\xi)\geq d(\xi).
\]
\end{proof}

\begin{proof}[Proof of $(\FC2)\Rightarrow (\FC3)$] 

We first note that (FC2), by the triangle inequality, implies the following condition:\\

\noindent\emph{For every finite, non-empty set $S\subseteq I$ and every $\varps> 0$ there exists a finite set $F\subseteq I$ such that}
\begin{align*}
\|\rho_{1,\chi_S}(\chi_F)-|S|\chi_F\|_{1,\sigma}<\varps \|\chi_F\|_{1,\sigma}.\tag{$\dagger$}
\end{align*}
Here $|S|$ denotes the cardinality of $S$.
Let $S$ and $\varps>0$ be given and choose $F$ such that $(\dagger)$ is satisfied. Define a map $\varphi\colon I\to \RR$ by   $\varphi(\xi)=\rho_{1,\chi_S}(\chi_F)(\xi)-|S|\chi_F(\xi)$. We note that
\begin{align*}
\varphi(\xi)&=\Big{(}\sum_{\alpha\in I}\chi_F(\alpha)(\delta_\xi\ast \chi_S)(\alpha)\Big{)}-|S|\chi_F(\xi)\\
&= \Big{(}\sum_{\alpha\in F}\sum_{\eta\in S}(\delta_\xi\ast\delta_\eta)(\alpha)\Big{)}-|S|\chi_F(\xi)\\
&=\sum_{\alpha\in F}\sum_{\eta\in S} \frac{d(\alpha)}{d(\xi)d(\eta)}N_{\xi,\eta}^\alpha  - |S|\chi_F(\xi). \\
\end{align*}
We now divide into four cases. 
\begin{itemize}
\item[(i)] If $\xi\in F\cap \del_S(F)^c $ then $\supp(\xi\eta)\subseteq F$ for all $\eta\in S$ and hence we get the relation $\sum_{\alpha\in F}\frac{d(\alpha)}{d(\xi)d(\eta)}N_{\xi,\eta}^\alpha=1$. This implies $\varphi(\xi)=0$.
\item[(ii)] If $\xi\in F^c\cap \del_S(F)^c$ we see that $N_{\xi,\eta}^\alpha=0$ for all $\alpha\in F$ and all $\eta\in S$ and hence $\varphi(\xi)=0$.
\item[(iii)] If $\xi\in F^c \cap \del_S(F)$ we have $\chi_F(\xi)=0$ and there exist $\alpha_0\in F$ and $\eta_0\in S$ such that   $N_{\xi,\eta_0}^{\alpha_0}\neq 0$. Using Lemma \ref{vurd-lem}, we now get
\[
\varphi(\xi)\geq \frac{d(\alpha_0)}{d(\xi)d(\eta_0)}N_{\xi,\eta_0}^{\alpha_0}\geq \frac{1}{d(\eta_0)^2}N_{\xi,\eta_0}^{\alpha_0}\geq \frac{1}{d(\eta_0)^2}\geq \frac{1}{M},
\]
where $M=\max\{d(\eta)^2\mid\eta\in S\}$.
\item[(iv)] If $\xi\in F\cap \del_S(F)$ we have
\begin{align*}
\varphi(\xi)& =\sum_{\alpha\in F}\sum_{\eta\in S}\frac{d(\alpha)}{d(\xi)d(\eta)}N_{\xi,\eta}^\alpha - |S|\\
&= (-1)\sum_{\eta\in S}\Big{(}1-\sum_{\alpha\in F} \frac{d(\alpha)}{d(\xi)d(\eta)}N_{\xi,\eta}^\alpha\Big{)}\\
&=(-1)\sum_{\eta\in S} \sum_{\alpha\notin F} \frac{d(\alpha)}{d(\xi)d(\eta)}N_{\xi,\eta}^\alpha,
\end{align*}
and because $\xi\in \del_S(F)\cap F$ there exist $\eta_0\in S$ and $\alpha_0\notin F$ such that $N_{\xi,\eta_0}^{\alpha_0}\neq 0$. Using Lemma \ref{vurd-lem} again we conclude, as in (iii), that $|\varphi(\xi)|\geq \frac{1}{M}$.
\end{itemize}
We now get
\begin{align*}
\varps\sum_{\xi\in F}d(\xi)^2&=\varps\|\chi_F\|_{1,\sigma}\\
&> \|\rho_{1,\chi_S}(\chi_F)-|S|\chi_F\|_{1,\sigma}\tag{by $(\dagger)$}\\
&=\sum_{\xi\in I}|\varphi(\xi)|d(\xi)^2\\
&=\sum_{\xi\in \del_S(F)}|\varphi(\xi)|d(\xi)^2 \tag{by (i) and (ii)}\\
&\geq \frac{1}{M}\sum_{\xi\in \del_S(F)}d(\xi)^2,\tag{by (iii) and (iv)}
\end{align*}
and since $\varps$ was arbitrary the claim follows.
\end{proof}

\begin{proof}[Proof of $(\FC 3)\Rightarrow (\FC 1)$] Given a finitely supported, symmetric pro\-ba\-bi\-lity measure $\mu$, with $\mu(e)>0$, and $\varps>0$ we put $S=\supp(\mu)$ and choose $F\subseteq I$ such that $(\FC 3)$ is fulfilled with respect to $\varps$. We have
\[
(\chi_F\ast \mu)(\xi) =\sum_{\alpha\in F,\beta\in S}\mu(\beta)\frac{d(\xi)}{d(\alpha)d(\beta)}N_{\alpha,\beta}^\xi,
\]
so
\begin{align*}
(\chi_F\ast \mu)(\xi)=0 &\Leftrightarrow \forall \alpha\in F\ \forall\beta\in S: N_{\alpha,\beta}^\xi =0\\
&\Leftrightarrow \forall \alpha\in F \ \forall\beta\in S: N_{\xi,\bar{\beta}}^\alpha =0\tag{Frobenius}\\
&\Leftrightarrow \forall \alpha\in F\ \forall\beta\in S: N_{\xi,\beta}^\alpha =0\tag{$S$ symmetric}\\
&\Leftrightarrow \xi\in F^c\cap \del_S(F)^c.\tag{$e\in S$}
\end{align*}
Hence $\supp(\chi_F\ast\mu)=(F^c\cap \del_S(F)^c)^c=F\cup\del_S(F)$ and we get
\begin{align*}
\sum_{\xi\in \supp(\chi_F\ast \mu)}d(\xi)^2-\sum_{\xi\in F}d(\xi)^2 &= \sum_{\xi\in F\cup\del_S(F)}d(\xi)^2-\sum_{\xi\in F}d(\xi)^2\\
&=\sum_{\xi\in \del_S(F)\cap F^c}d(\xi)^2\\
&\leq \sum_{\xi\in \del_S(F)}d(\xi)^2\\
&<\varps \sum_{\xi\in F}d(\xi)^2. \tag{by $(\FC3)$} 
\end{align*}
\end{proof}
\begin{proof}[Proof of $(\FC1)\Rightarrow (\FC2)$]
Given $\varps>0$ and $S\subseteq I$ we define $\tilde{S}=S\cup \bar{S}\cup\{e\}$ and $\mu=\frac{1}{|\tilde{S}|}\chi_{\tilde{S}}$. Choose $F\subseteq I$ such that $\mu$ and $F$ satisfy $(\FC 1)$ with respect to $\frac{\varps}{2}$. We aim to prove that $(\FC 2)$ is satisfied for all $\xi\in \tilde{S}$.  For arbitrary $\xi\in I$ we have
\begin{align*}
\|\rho_{1,\xi}(\chi_F)-\chi_F\|_{1,\sigma} &= \sum_{\alpha}|\rho_{1,\xi}(\chi_F)(\alpha)-\chi_F(\alpha)|d(\alpha)^2\\
&=\sum_{\alpha}|(\sum_{\eta\in F}\frac{d(\eta)}{d(\alpha)d(\xi)}N_{\alpha,\xi}^\eta)-\chi_F(\alpha)|d(\alpha)^2\\
&=\sum_{\alpha\in F}\Big{(}1-\sum_{\eta\in F}\frac{d(\eta)}{d(\alpha)d(\xi)}N_{\alpha,\xi}^\eta\Big{)}d(\alpha)^2\\
&\quad+\sum_{\alpha\notin F}\Big{(}\sum_{\eta\in F}\frac{d(\eta)}{d(\alpha)d(\xi)}N_{\alpha,\xi}^\eta\Big{)}d(\alpha)^2\\
&=\sum_{\alpha\in F}\sum_{\eta\notin F}\frac{d(\eta)d(\alpha)}{d(\xi)}N_{\alpha,\xi}^\eta
+\sum_{\alpha\notin F}\sum_{\eta\in F}\frac{d(\eta)d(\alpha)}{d(\xi)}N_{\alpha,\xi}^\eta\\
&=\sum_{\alpha\notin F}\sum_{\eta\in F}\frac{d(\eta)d(\alpha)}{d(\xi)}(N_{\alpha,\xi}^\eta +N_{\eta,\xi}^\alpha)\\
&=\sum_{\alpha\notin F}\sum_{\eta\in F}\frac{d(\eta)d(\alpha)}{d(\xi)}(N_{\eta,\bar{\xi}}^\alpha +N_{\eta,\xi}^\alpha).\tag{$\dagger$}
\end{align*}
For  $\xi\in \supp(\mu)=\tilde{S}$ and $\alpha\notin F$, it is easy to check that $(\chi_F\ast \mu)(\alpha)> 0$ if there exists an $\eta\in F$ such that $N_{\eta,\bar{\xi}}^\alpha +N_{\eta,\xi}^\alpha>0$. Hence the calculation $(\dagger)$ implies that
\begin{align*}
\|\rho_{1,\xi}(\chi_F)-\chi_F\|_{1,\sigma} &\leq \sum_{\alpha\in \supp(\chi_F\ast\mu)\setminus F}\sum_{\eta\in F}\frac{d(\eta)d(\alpha)}{d(\xi)}(N_{\eta,\bar{\xi}}^\alpha +N_{\eta,\xi}^\alpha)\\
&\leq \sum_{\alpha\in \supp(\chi_F\ast\mu)\setminus F}\sum_{\eta\in I}\frac{d(\eta)d(\alpha)}{d(\xi)}(N_{\eta,\bar{\xi}}^\alpha +N_{\eta,\xi}^\alpha)\\
&= 2\sum_{\alpha\in \supp(\chi_F\ast\mu)\setminus F}d(\alpha)^2\\
&=2\Big{(} \sum_{\alpha\in \supp(\chi_F\ast\mu)}d(\alpha)^2- \sum_{\alpha\in  F}d(\alpha)^2   \Big{)}\\
&< \varps \|\chi_F\|_{1,\sigma},%\tag{by $(\FC 1)$}
\end{align*}
where the last estimate follows from $(\FC 1)$. Note that the condition $e\in \supp(\mu)$ was used to get the fourth step in the calculation above.
\end{proof}

We now set out to prove the  remaining equivalence in Theorem \ref{izumi-thm}.

\begin{proof}[Proof of $(\A)\Leftrightarrow (\FC 2)$] At the end of this section four formulas are gathered; these will be used during the proof and referred to as (F1) - (F4). For the actual proof we also need the following definitions. Consider a finitely supported, symmetric probability measure $\mu$ on $I$ and define $p_\mu\colon I\times I\to \RR$ by
\[
p_\mu(\xi,\eta)=(\delta_\xi\ast \mu)(\eta)=\sum_{\omega}\mu(\omega)\frac{d(\eta)}{d(\xi)d(\omega)}N_{\xi,\omega}^\eta.
\]
Note that the function $p_\mu$ satisfies the \emph{reversibility condition}:
\[
\sigma(\xi)p_\mu(\xi,\eta)=\sigma(\eta)p_\mu(\eta,\xi).
\]
For a finitely supported function $f\in c_0(I)$ and $r\in \NN$ we also define
\[
\|f\|_{D_\mu(r)}=\Big{(}\frac12\sum_{\xi,\eta}\sigma(\xi)p_\mu(\xi,\eta)|f(\xi)-f(\eta)|^r\Big{)}^{\frac{1}{r}}.
\]
Although this is referred to as the \emph{generalized Dirichlet $r$-norm} of $f$, one should keep in mind that the function $\|\cdot\|_{D_\mu(r)}$ is only a semi norm.  We shall now consider the following condition:\\

\noindent\emph{For all finitely supported, symmetric, probability measures $\mu$ we have }
\begin{align*}
\inf\Big{\{}\frac{\|f\|_{D_\mu(r)}}{\|f\|_{r,\sigma}} \mid f\in c_0(I)\setminus\{0\} \Big{\}}=0.\tag{NW$_r$}
\end{align*}
The reason for the name (NW$_r$), which appeared in \cite{izumi}, is that the condition is the negation of a so-called Wirtinger inequality. See \cite{izumi} for more details. To prove $(\A)\Leftrightarrow (\FC 2)$ we will actually prove the following equivalences
\[
(\FC2)\Leftrightarrow (\NW_1) \hspace{8pt} \text{ and } \hspace{8pt} \forall r: (\NW_1)\Leftrightarrow  (\NW_r) \hspace{8pt} \text{ and } \hspace{8pt} (\tetxrm{A})\Leftrightarrow (\NW_2). 
\]
For the latter of these equivalences the following lemma will be useful.
\begin{lem}\label{ip-og-dirichlet}
For all $f\in c_0(I)$ we have 
\[
\|f\|_{D_\mu(2)}^2=\ip{f}{f}_{2,\sigma}-\ip{\rho_{2,\mu}(f)}{f}_{2,\sigma},
\]
where $\ipp_{2,\sigma}$ denotes the inner product on $\ell^2(I,\sigma)$.
\end{lem}
\begin{proof}
This is proven by a direct calculation using the reversibility condition and the formula (F4) from the end of this section.
\end{proof}
\begin{proof}[Proof of \emph{(A)$\Leftrightarrow (\NW_2)$}]
Let $\mu$ be a finitely supported, symmetric probability measure on  $I$.
By \cite[1.3,1.5]{izumi}, we have that $\rho_{2,\mu}$ is self-adjoint and $\|\rho_{2,\mu}\|\leq \|\mu\|_1=1$ so that $1-\rho_{2,\mu}\geq 0$. We now get
\begin{align*}
1\in \sigma(\lambda_{2,\mu}) &\Leftrightarrow 1\in \sigma(\rho_{2,\mu})\tag{\cite[1.5]{izumi}}\\
& \Leftrightarrow 0\in \sigma(1-\rho_{2,\mu})\\
&\Leftrightarrow 0\in \sigma(\sqrt{1-\rho_{2,\mu}})\\
&\Leftrightarrow \exists x_n\in (\ell^2(I,\sigma))_1: \|(\sqrt{1-\rho_{2,\mu}})x_n\|_{2,\sigma}\To 0\\ %\tag{\cite[3.2.13]{KR1}}\\
&\Leftrightarrow \exists f_n\in (c_0(I))_1: \|(\sqrt{1-\rho_{2,\mu}})f_n\|_{2,\sigma}\To 0 \\
&\Leftrightarrow \exists f_n\in (c_0(I))_1: \ip{(1-\rho_{2,\mu})f_n}{f_n}_{2,\sigma}\To 0\\
&\Leftrightarrow \exists f_n\in (c_0(I))_1: \|f_n\|_{D_\mu(2)}\To 0\tag{Lem. \ref{ip-og-dirichlet}}\\
&\Leftrightarrow  \inf\Big{\{}\frac{\|f\|_{D_\mu(2)}}{\|f\|_{2,\sigma}} \mid f\in c_0(I)\setminus\{0\} \Big{\}}=0.
\end{align*}
Hence $(\A)\Leftrightarrow (\NW_2)$ as desired.
\end{proof}
\begin{proof}[Proof of $(\NW_1) \Rightarrow (\FC2) $]
Given $\varps>0$ and $\xi_1,\dots,\xi_n\in I$, we choose a finitely supported, symmetric probability measure $\mu$ with $\xi_1,\dots,\xi_n\in \supp(\mu)$. Define 
\[
\varps'=\frac{\varps}{2}\min\{\mu(\xi)\mid\xi\in I \},
\]
and choose, according to $(\NW_1)$, an $f\in c_0(I)$ such that 
\begin{align*}
\|f\|_{D_\mu(1)}<\varps'\|f\|_{1,\sigma}.\tag{$\ast$}
\end{align*}
Since $\||f|\|_{D_\mu(1)}\leq \|f\|_{D_\mu(1)}$ and $\||f|\|_{1,\sigma}=\|f\|_{1,\sigma}$ we may assume that $f$ is positive. Since $f$ can be approximated by a rational function we may actually assume that $f$ has integer values. Put $N=\max\{f(\xi)\mid \xi\in I\}$ and define,  for $k=1,\dots, N$, $F_k=\{\xi\mid f(\xi)\geq k\}$. Then $f=\sum_{k=1}^N\chi_{F_k}$ and the following formulas hold.
\[
\|f\|_{D_\mu(1)}=\sum_{k=1}^N \|\chi_{F_k}\|_{D_\mu(1)} \quad \textrm{ and } \quad \|f\|_{1,\sigma}=\sum_{k=1}^N \|\chi_{F_k}\|_{1,\sigma}. 
\]
The first formula is proved by induction on the integer $N$ and the second follows from a direct calculation using only the reversibility property of $p_\mu$. Because of $(\ast)$, there must therefore exist some $j\in \{1,\dots, N\}$ such that 
\begin{align*}
\|\chi_{F_j}\|_{D_\mu(1)}<\varps' \|\chi_{F_j}\|_{1,\sigma}.\tag{$\ast\ast$}
\end{align*}
For the sake of simplicity we denote this $F_j$ by $F$ in the following. We now get
\begin{align*}
\|\chi_F\|_{D_\mu(1)} &= \frac12 \sum_{\xi,\eta}\sigma(\xi)p_\mu(\xi,\eta)|\chi_F(\xi)-\chi_F(\eta)|\\
&= \sum_{\xi\in F,\eta\notin F}\sigma(\xi)p_\mu(\xi,\eta)\tag{reversibility}\\
&=\sum_{\xi\in F,\eta\notin F}\sigma(\xi)\Big{(}\sum_\omega\mu(\omega)\frac{d(\eta)}{d(\xi)d(\omega)}N_{\xi,\omega}^\eta\Big{)}\\
&=\sum_{\omega}\mu(\omega)\Big{(}\sum_{\xi\in F,\eta\notin F}\frac{d(\xi)d(\eta)}{d(\omega)}N_{\xi,\omega}^\eta\Big{)}\\
&=\frac12\sum_{\omega}\mu(\omega)\Big{(}\sum_{\xi\in F,\eta\notin F}\frac{d(\xi)d(\eta)}{d(\omega)}(N_{\xi,\omega}^\eta+N_{\xi,\bar{\omega}}^\eta)\Big{)}\\
&=\frac12 \sum_\omega \mu(\omega)\|\rho_{1,\omega}(\chi_F)-\chi_F\|_{1,\sigma}.\tag{$\ddagger$}
\end{align*}
%Ligeningen over (dobbeltdagger) fï¿½s ved at skrive den foregï¿½ende sum op to gange og og lade index i den sidste lï¿½be over \bar{\omega} i stedet for \omega. Idet \mu er symmetrisk fï¿½s det ï¿½nskede udtryk.
Here the last equality follows from the computation $(\dagger)$ in the proof of $(\FC1)\Rightarrow (\FC 2)$.
The inequality $(\ast\ast)$ therefore reads
\[
\frac12 \sum_\omega \mu(\omega)\|\rho_{1,\omega}(\chi_F)-\chi_F\|_{1,\sigma}<\varps'\|\chi_F\|_{1,\sigma}.
\]
For every $\omega\in I$ we therefore conclude, since $\varps'=\frac{\varps}{2} \min(\mu)$, that
\[
\mu(\omega)\|\rho_{1,\omega}(\chi_F)-\chi_F\|_{1,\sigma}<\min(\mu)\varps \|\chi_F\|_{1,\sigma}.
\]
Since each of the given $\xi_i$'s are in $\supp(\mu)$ we get for all $i$ that
\[
\|\rho_{1,\xi_i}(\chi_F)-\chi_F\|_{1,\sigma}<\varps \|\chi_F\|_{1,\sigma},
\]
as desired.
\end{proof}

\begin{proof}[Proof of $(\FC2) \Rightarrow (\NW_1)$]
Assume now $(\FC2)$ and let $\mu$ and $\varps$ be given. Choose $F$ such that
\[
\|\rho_{1,\xi}(\chi_F)-\chi_F\|_{1,\sigma}<\varps \|\chi_F\|_{1,\sigma}
\]
for all $\xi\in \supp(\mu)$. Using the calculation $(\ddagger)$, from the proof of opposite implication, we get
\begin{align*}
\|\chi_F\|_{D_\mu(1)}&=\frac12 \sum_{\omega}\mu(\omega)\|\rho_{1,\omega}(\chi_F)-\chi_F\|_{1,\sigma}\\
&<\frac12 \sum_\omega \mu(\omega)\varps\|\chi_F\|_{1,\sigma}\\
&=\frac{\varps}{2}\|\chi_F\|_{1,\sigma}\\
&<\varps\|\chi_F\|_{1,\sigma}.
\end{align*}
\end{proof}

For the proof of the statement $(\NW_1)\Leftrightarrow (\NW_r)$ we will need the following lemma.
\begin{lem}[\cite{gerl}]\label{uliglem}
For $r\geq 2$ and $f\in c_0(I)_+$ we have 
\[
\|f^r\|_{D_\mu(1)}\leq 2r\|f\|_{r,\sigma}^{r-1}\|f\|_{D_\mu(r)}.
\]
\end{lem}

\begin{proof}
First note that
\begin{align*}
\|f^r\|_{D_\mu(1)} &=\frac12 \sum_{\xi,\eta}\sigma(\xi)p_\mu(\xi,\eta)|f(\xi)^r-f(\eta)^r|\\
&\leq \frac{r}{2}\sum_{\xi,\eta}\sigma(\xi)p_\mu(\xi,\eta)(f(\xi)^{r-1}+f(\eta)^{r-1})|f(\xi)-f(\eta)|,
\end{align*}
where the inequality follows from (F1).
Define a measure $\nu$ on $I\times I$ by $\nu(\xi,\eta)=\frac12 \sigma(\xi)p_\mu(\xi,\eta)$  and consider the functions $\varphi,\psi\colon I\times I\to \RR$ given by
\begin{align*}
\varphi(\xi,\eta)&=f(\xi)^{r-1} + f(\eta)^{r-1}\quad \text{ and }\quad 
\psi(\xi,\eta)=|f(\xi)-f(\eta)|.
\end{align*}
Define $s>1$ by the equation $\frac{1}{r}+\frac{1}{s}=1$. Then the inequality above can be written as $\|f^r\|_{D_\mu(1)} \leq r \|\varphi\psi\|_{1,\nu}$ and using H\"older's inequality we therefore get
{\allowdisplaybreaks
\begin{align*}
\|f^r\|_{D_\mu(1)}& \leq r \|\varphi\psi\|_{1,\nu}\\ 
&\leq r\|\varphi\|_{s,\nu}\|\psi\|_{r,\nu}\\
&= r\Big{[}\sum_{\xi,\eta} \frac12 \sigma(\xi)p_\mu(\xi,\eta)(f(\xi)^{r-1}+f(\eta)^{r-1})^s\Big{]}^{\frac{1}{s}}\\
&\hspace{0.39cm} \times \Big{[}\sum_{\xi,\eta}\frac12 \sigma(\xi)p_\mu(\xi,\eta)|f(\xi)-f(\eta)|^r\Big{]}^{\frac{1}{r}}\\
&\stackrel{\footnotemark[1]}{\leq} r \Big{[}2^{s-1}\frac12 \sum_{\xi,\eta}\sigma(\xi)p_\mu(\xi,\eta)(f(\xi)^{(r-1)s}+f(\eta)^{(r-1)s})\Big{]}^{\frac{1}{s}}\|f\|_{D_\mu(r)}\\%\tag{by(F2)}\\
&\stackrel{\footnotemark[2]}{=}r \Big{[}2^{s-1}\sum_{\xi,\eta}\sigma(\xi)p_\mu(\xi,\eta)f(\xi)^{(r-1)s}\Big{]}^{\frac{1}{s}}\|f\|_{D_\mu(r)}\\%\tag{reversibility}\\
&=r2^{\frac{s-1}{s}}\Big{[}\sum_\xi \sigma(\xi)\Big{(}\sum_{\eta}p_\mu(\xi,\eta)\Big{)}f(\xi)^{(r-1)s}\Big{]}^{\frac{1}{s}}\|f\|_{D_\mu(r)}\\
&=r2^{\frac{s-1}{s}}\Big{[}\sum_\xi \sigma(\xi)f(\xi)^{(r-1)s}\Big{]}^{\frac{1}{s}}\|f\|_{D_\mu(r)}\\
&\leq 2r\Big{[}\sum_\xi \sigma(\xi)f(\xi)^r\Big{]}^{\frac{r-1}{r}}\|f\|_{D_\mu(r)}\\
&=2r\|f\|_{r,\sigma}^{r-1}\|f\|_{D_\mu(r)}.
\end{align*}}
\footnotetext[1]{by (F2)}
\footnotetext[2]{by reversibility}

\end{proof}
Also the following observation will be useful.
\begin{obs}\label{uligobs}
Under the assumptions of Lemma \ref{uliglem} we have 
\begin{align*}
\|f\|_{D_\mu(r)} &=\Big{[}\frac12 \sum_{\xi,\eta}\sigma(\xi)p_\mu(\xi,\eta)|f(\xi)-f(\eta)|^{r}    \Big{]}^\frac{1}{r}\\
&\leq \Big{[}\frac12 \sum_{\xi,\eta}\sigma(\xi)p_\mu(\xi,\eta)|f(\xi)^r-f(\eta)^r|    \Big{]}^\frac{1}{r}\tag{by (F3)}\\
&=\|f^r\|_{D_\mu(1)}^{\frac{1}{r}}.
\end{align*}
\end{obs}
%With these results available it is easy to prove the equivalence $(\NW_1)\Leftrightarrow (\NW_r)$.
Having these results, we are now able to prove $(\NW_1)\Leftrightarrow (\NW_r)$.
\begin{proof}[Proof of $(\NW_1)\Rightarrow (\NW_r)$.]\hspace{0.01cm}Assume $(\NW_1)$ and let $\mu$ and $\varps>0$ be given. Put $\varps'=\varps^r$ and choose non-zero  $f\in c_0(I)_+$ such that
\[
\frac{\|f\|_{D_\mu(1)}}{\|f\|_{1,\sigma}}<\varps'.
\]
Using Observation \ref{uligobs} we get
\[
\frac{\|\sqrt[r]{f}\|_{D_\mu(r)}}{\|\sqrt[r]{f}\|_{r,\sigma}}\leq \frac{\|f\|_{D_\mu(1)}^{\frac{1}{r}}}{\|f\|_{1,\sigma}^{\frac{1}{r}}}<(\varps')^{\frac{1}{r}}=\varps.
\]
\end{proof}

\begin{proof}[Proof of $(\NW_r)\Rightarrow (\NW_1)$]
Given $\mu$ and $\varps>0$ and put $\varps'=\frac{1}{2r}\varps$. Then choose non-zero $f\in c_0(I)_+$ with
\[
\frac{\|f\|_{D_\mu(r)}}{\|f\|_{r,\sigma}}<\varps'.
\]
Using Lemma \ref{uliglem}, we get
\[
\frac{\|f^r\|_{D_\mu(1)}}{\|f^r\|_{1,\sigma}}\leq \frac{2r\|f\|_{r,\sigma}^{r-1}\|f\|_{D_\mu(r)}}{\|f\|_{r,\sigma}^r} <2r\varps '=\varps.
\]
\end{proof}
Gathering all the results just proven we get $(\A)\Leftrightarrow (\FC2)$.
\end{proof}
This concludes the proof of Theorem \ref{izumi-thm}.

\begin{rem}\label{foelner-rem}
Consider a countable, discrete group $\Gamma$ and the corresponding fusion algebra $\ZZ[\Gamma]$. It is not difficult to prove that $\ZZ[\Gamma]$ satisfies $(\FC 3)$ from Theorem \ref{izumi-thm} if and only if $\Gamma$ satisfies F{\o}lner's condition (for groups) as presented in \cite[F.6]{benedetti}. Since a group is amenable if and only if it satisfies F{\o}lner's condition, we see from this that $\Gamma$ is amenable if and only if the corresponding fusion algebra $\ZZ [\Gamma]$ is amenable. 
\end{rem}

\subsection{Formulas used in the proof of Theorem \ref{izumi-thm}}
We collect here four formulas used in the proof of Theorem \ref{izumi-thm}. Let $r,s>1$ and assume that $\frac{1}{r}+\frac{1}{s}=1$. Then for all $z,w\in \CC$, $a,b\geq 0$ and $n\in \NN$ we have
\begin{align*}
|a^r-b^r| &\leq r(a^{r-1}+b^{r-1})|a-b| \tag{F1}\\
(a+b)^r&\leq 2^{r-1}(a^r+b^r)\tag{F2}\\
|a-b|^n&\leq |a^n-b^n|\tag{F3}\\
|z-w|^2+|w-z|^2 &=2(|z|^2-z\bar{w}) +2(|w|^2-w\bar{z})\tag{F4}
\end{align*}
\begin{proof}
The inequality (F1) can be proved using the mean value theorem on the function $f(x)=x^r$ and the interval between $a$ and $b$. To prove (F2), consider a two-point set endowed with counting measure. Using H\"older's inequality, we then get
\[
a+b=1\cdot a + 1\cdot b\leq (1^s +1^s)^{\frac{1}{s}}(a^r+b^r)^{\frac{1}{r}}.
\]
From this the desired inequality follows using the fact that $\frac{1}{s}=\frac{r-1}{r}$. The inequality (F3) follows using the binomial theorem. If, for instance, $a=b+k$ for some $k\geq 0$ we have
\[
(a-b)^n=k^n\leq (b+k)^n - b^n=a^n-b^n.
\]
The formula (F4) follows by splitting $w$ and $z$ into real and imaginary parts and calculating both sides of the equation.
\end{proof}

\section{Coamenable Compact Quantum Groups}\label{sectionfour}
In this section we introduce the notion of coamenability for compact quantum groups and discuss the relationship between coamenability of a compact quantum group and amenability of its corepresentation ring. The notion of (co-)amenability has been treated in different quantum group settings by numerous people. A number of references for this subject are \cite{murphy-tuset}, \cite{voiculescu-amenability}, \cite{ruan-amenability}, \cite{banica-fusion}, \cite{banica-subfactor}, \cite{enock-schwartz} and \cite{baaj-skandalis}. For our purposes, the approach of B{\'e}dos, Murphy and Tuset in \cite{murphy-tuset} is the most natural and we are therefore going to follow this reference throughout this section. We will assume that the reader is familiar with the basics on Woronowicz's theory of compact quantum groups. Definitions, notation and some basic properties can be found in Section \ref{prelim} and  detailed treatments can be found in \cite{wor-cp-qgrps}, \cite{vandaele} and \cite{tuset}. 
\begin{defi}[\cite{murphy-tuset}]\label{coamenable-defi}
Let $\GG=(A,\Delta)$ be a compact quantum group and let $A_\red$ be the image of $A$ under the GNS representation $\pi_h$ arising from the Haar state $h$. Then $\GG$ is said to be coamenable if the counit $\varps\colon A_0\to \CC$ extends continuously to $A_\red$.
\end{defi}
\begin{rem}
It is well known that a discrete group $\Gamma$ is amenable if and only if the trivial representation of $C^*_{\full}(\Gamma)$ factorizes through $C^*_\red(\Gamma)$. This amounts to saying that $(C^*_\red(\Gamma),\Delta_\red)$ is coamenable if and only if $\Gamma$ is amenable. Note also that the abelian compact quantum groups $(C(G),\Delta_c)$ are automatically coamenable since the counit is given by evaluation at the identity and therefore already globally defined and bounded.
\end{rem}
In the following theorem we collect some facts on coamenable compact quantum groups. For more coamenability criteria and a proof of the theorem below we refer to \cite{murphy-tuset}.
\begin{thm}[\cite{murphy-tuset}]\label{murphy-tuset-thm}
For a compact quantum group $\GG=(A,\Delta)$ the following are equivalent.
\begin{itemize}
\item[(i)] $\GG$ is coamenable.
\item[(ii)] The Haar state $h$ is faithful and the counit is bounded with respect to the norm on $A$.
\item[(iii)] The natural map from the universal representation $A_{\textrm{u}}$ to the reduced representation $A_\red$ is an isomorphism.
\end{itemize}
If $\GG$ is a compact matrix quantum group with fundamental corepresentation $u\in \MM_n(A)$ the above conditions are also equivalent to the following.
\begin{itemize}
\item [(iv)]  The number $n$ is in $\sigma(\pi_h(\Re(\chi(u)))$ where $\chi(u)=\sum_{i=1}^n u_{ii}$ is the character map from Section \ref{fusion-section}.
\end{itemize}
\end{thm}
Recall that $\sigma(T)$ denotes the spectrum of a given operator $T$. Thus, when we are dealing with a coamenable quantum group the Haar state is automatically faithful and hence the corresponding GNS representation $\pi_h$ is faithful. We therefore can, and will, identify $A$ and $A_\red$. The condition (iv) is Skandalis's quantum analogue of the so-called Kesten condition for groups (see \cite{kesten} and \cite{banica-fusion}) which is proved by Banica in \cite{banica-subfactor}. The next result is a generalization of the Kesten condition to the case where a fundamental corepresentation is not (ne\-cessarily) present. The proof draws inspiration from the corresponding proof in \cite{murphy-tuset}.
\begin{thm}\label{kesten-prop}
Let $\GG=(A,\Delta)$ be a compact quantum group. Then the following are equivalent:
\begin{itemize}
\item[(i)] $\GG$ is coamenable.
\item[(ii)] For any finite dimensional, unitary co\-re\-pre\-sen\-tation $u\in \MM_{n_u}(A)$ we have $n_u\in \sigma (\pi_h(\Re(\chi(u))))$.
\end{itemize}
\end{thm}
\begin{proof}
Assume $\GG$ to be coamenable and let a finite dimensional, unitary corepresentation $u\in \MM_{n_u}(A)$ be given. Since the counit extends to a character $\varps\colon A_\red\to \CC$ and since
\[
\varps(\Re(\chi(u)))=\varps(\sum_{i=1}^{n_u} \frac{u_{ii} +u_{ii}^*}{2})=n_u,
\]
we must have $n_u\in \sigma(\pi_h(\Re(\chi(u))))$. Assume conversely that the property (ii) is satisfied and define, for a finite dimensional, unitary corepresentation $u$, the set
\[
C(u)=\{\varphi\in \S(A_\red)\mid \varphi(\pi_h(\Re(\chi(u))))=n_u\}.
\]
Here $\S(A_\red)$ denotes the state space of $A_\red$. It is clear that each $C(u)$ is closed in the weak$^*$-topology and we now prove that the family 
\[
\mathcal{F}=\{C(u)\mid u \text{ finite dimensional, unitary corepresentation}\} 
\]
has the finite intersection property. We first prove that each $C(u)$ is non-empty. For given $u$, we put $x_{ij}=u_{ij}-\delta_{ij}$ and $x=\sum_{ij} x_{ij}^*x_{ij}$. Then $x$ is clearly positive and a direct calculation reveals that
\begin{align*}
x=2(n_u-\Re(\chi(u))).\tag{$\dagger$}
\end{align*}
Hence, $n_u\in \sigma(\pi_h(\Re(\chi(u))))$ if and only if there exists \cite[4.4.4]{KR1} a $\varphi\in \S(A_\red)$ with
\[
\varphi(\pi_h(\Re(\chi(u))))=n_u.
\]
Thus, $C(u)\neq \emptyset$. Let now $u^{(1)},\dots, u^{(k)}$ be given and put $u=\oplus_{i=1}^k u^{(i)}$. We aim at proving that
\[
C(u)\subseteq \bigcap_{i=1}^k C(u^{(i)}).
\]
Let $\varphi\in C(u)$ be given and note that
\begin{align*}
\sum_{i=1}^k n_{u^{(k)}}=\varphi(\pi_h(\Re(\chi(u))))
=\sum_{i=1}^k\sum_{j=1}^{n_{u^{(i)}}}\frac12 \varphi(\pi_h(u^{(i)}_{jj}) +\pi_h({u^{(i)*}_{jj}})).
\end{align*}
Since the matrix  $u$ is unitary, we have $\|\pi_{h}(u_{st}) \|\leq 1$ for all $s,t\in\{1,\dots, n_u\}$ 
%the unitarity of $u$ forces $\sum_j u_{ij}u_{ij}^* $ for all $i$. Thus we have $n$ positive elements summing to $1$ and they must therefore all be less than 1. That is, $u_{ij}u_{ij}^*\leq 1$ and therefore $\|u_{ij}\|^2=\|u_{ij}u_{ij}^*\|\leq 1$.
and hence
\[
\frac12 \varphi(\pi_h(u^{(i)}_{jj}) +\pi_h({u^{(i)*}_{jj}}))\in [-1,1].
\]
This forces $\frac12 \varphi(\pi_h(u^{(i)}_{jj}) +\pi_h({u^{(i)*}_{jj}}))=1$ and hence $\varphi(\pi_h(\Re(\chi(u^{(i)}))))=n_{u^{(i)}}$. Thus $\varphi$ is in each of the sets $C(u^{(1)}),\dots, C(u^{(k)})$ and we conclude that $\mathcal{F}$ has the finite intersection property. By compactness of $\S(A_\red)$, we may therefore find a state $\varphi$ such that $\varphi(\pi_h(\Re(\chi(u))))=n_u$ for every unitary corepresentation $u$. Denote by $H$ the GNS space associated with this $\varphi$, by $\xi_0$ the natural cyclic vector and by $\pi$ the corresponding GNS 
representation of $A_\red$. Consider an arbitrary unitary corepresentation $u$ and form as before the elements $x_{ij}$ and $x$. Then the equation $(\dagger)$ shows that $\varphi(x_{ij}^*x_{ij})=0$ and hence $\pi(x_{ij})\xi_0=0$ and
\[
\pi(u_{ij})\xi_0=\delta_{ij}\xi_0.
\]
From the Cauchy-Schwarz inequality we get 
\[
|\varphi(x_{ij})|^2\leq \varphi(x_{ij}^*x_{ij})\varphi(1)=0,
\]
and hence $\varphi(u_{ij})=\delta_{ij}$. We therefore have that $\pi(u_{ij})\xi_0=\varphi(u_{ij})\xi_0$. Since the matrix coefficients span $A_0$ linearly we get $\pi(a)\xi_0=\varphi(a)\xi_0$ for all $a\in A_0$. By density of $A_0$ in $A_\red$ it follows that $\pi(a)\xi_0=\varphi(a)\xi_0$ for all $a\in A_\red$. From this we see that
\[
H=\overline{\pi(A_\red)\xi_0}^{\|\cdot\|_2}=\CC\xi_0,
\]
and it follows that $\varphi\colon A_{\red}\to \CC$ is a bounded $*$-homomorphism coinciding with $\varps$ on $A_0$. Thus, $\GG$ is coamenable.

\end{proof}

The following result was mentioned, without proof, in \cite[p.692]{izumi} in the restricted setting of compact matrix quantum groups whose Haar state is a trace.

\begin{thm}\label{quantum-fusion-link}
A compact quantum group $\GG=(A,\Delta)$ is a coamenable if and only if the corepresentation ring $R(\GG)$ is amenable.
\end{thm}
For the proof we will need the following lemma. For this, recall from Section \ref{fusion-section} that the $*$-algebra $\CC[\Irred(\GG)]$ comes with a  trace $\tau$ given by
\[
\sum_{u\in \Irred(\GG)}z_u u\longmapsto z_{e},
\] 
where $e\in \Irred(\GG)$ denotes the identity in $R(\GG)$. In what follows, we denote by $C^*_\red(R(\GG))$ the enveloping $C^*$-algebra of $\CC[\Irred(\GG)]$ on the GNS space $L^2(\CC[\Irred(\GG)],\tau)$ arising from $\tau$. 
\begin{lem}\label{char-extends}
The character map $\chi\colon R(\GG)\to A_0$ extends to an iso\-me\-tric $*$-homomorphism $\chi\colon C^*_\red(R(\GG))\to A_\red$.
\end{lem}
%BEMï¿½RK, VI SKAL IKKE BRUGE COAMENABILITET HER - EJ HELLER AT $h$ ER ET SPOR. !!!!!!!
\begin{proof}
Put $I=\Irred(\GG)$. For an irreducible, finite dimensional, unitary corepresentation $u$ we have $h(u_{ij})=0$ unless $u$ is the trivial corepresentation and therefore  the following diagram commutes
\[
\xymatrix{ 
\CC[I]\ar[d]_{\tau} \ar@{^{(}->}[rr]^{\chi} & & A_0\ar[dll]^{h}\\
\CC & &  }
\]
Hence $\chi$ extends to an isometric embedding 
\[
K=L^2(\CC[I],\tau)\hooklongrightarrow L^2(A_0,h)=H.
\]
Denote by $S$ the algebra $\chi(R(\GG))$ and by $\bar{S}$ the closure of $\pi_h(S)$ inside $A_\red$. Since $S$ is a $*$-algebra that maps $K$ into itself it also maps $K^\bot$ into itself and hence $\pi_h(\chi(a))$ takes the form
\begin{align*}
\left(%
\begin{array}{cc}
\pi_h(\chi(a)) \big{|}_K & 0 \\
0 & \pi_h(\chi(a))\big{|}_{K^\bot}
\end{array}%
\right).
\end{align*}
Thus 
\begin{align*}
\|\pi_h(\chi(a))\| &=\max\{\|\pi_h(\chi(a))\big{|}_{K}\|,\|\pi_h(\chi(a))\big{|}_{K^\bot}\|\}\\
&\geq \|\pi_h(\chi(a))\big{|}_{K}\|\\
&=\|\pi_\tau (a)\|.
\end{align*}
This proves that the map $\kappa\colon \pi_h(S)\to \pi_\tau(\CC[I])$ given by $\kappa(\pi_h(\chi(a)))=\pi_\tau(a)$ is bounded and it therefore extends to a contraction $\bar{\kappa}\colon \bar{S}\to C^*_\red(R(\GG))$. We now prove that $\bar{\kappa}$ is injective. Since $h$ is faithful on $A_\red$ and $\tau$ is faithful on $C^*_\red(R(\GG))$ we get the following commutative diagram
\[
\xymatrix{ 
\pi_h(S) \ar[rr]^{\kappa}_{\sim} \ar@{^{(}->}[d] & & \pi_\tau(\CC[I])\ar@{^{(}->}[d]\\
\bar{S} \ar[rr]^{\bar{\kappa}} \ar@{^{(}->}[d] && C^*_\red(R(\GG)) \ar@{^{(}->}[d]\\
L^2(\bar{S},h) \ar[rr] && L^2(C^*_\red(R(\GG)),\tau) 
}
\]
One easily checks that $\kappa$ induces an isometry $L^2(\bar{S},h)\to L^2(C^*_\red(\GG),\tau)$ and it therefore follows that $\bar{\kappa}$ is injective and hence an isometry. Thus, for $\chi(a)\in S$ we have
\[
\|\pi_h(\chi(a))\|=\|\bar{\kappa}(\pi_h(\chi(a)))\|=\|\pi_\tau(a)\|,
\]
as desired.

\end{proof}

\begin{proof}[Proof of Theorem \ref{quantum-fusion-link}]
Assume first that $\GG$ is coamenable and put $I=\Irred(\GG)$. Consider a finitely supported, symmetric probability measure $\mu$ on $I$.  We aim to show that $1\in \sigma(\lambda_{2,\mu})$, where $\lambda_{2,\mu}$ is the operator on $\ell^2(I,\sigma)$ defined in Section \ref{fusion-section}. Write $\mu$ as $\sum_{\xi\in I}t_\xi\delta_\xi$ and recall  (Lemma \ref{char-extends}) that the character map $\chi\colon \CC[I]\to A_0$ extends to an injective $*$-homomorphism $\chi\colon C^*_\red(R(\GG))\to A_\red$. Using this, and Proposition \ref{gns-and-lambda}, we get that
\begin{align*}
\sigma(\lambda_{2,\mu})&=\sigma(l_\mu)\\
&=\sigma(\sum_{\xi\in I} t_\xi l_\xi)\\
&= \sigma(\sum_{\xi\in I} t_\xi \frac{1}{n_\xi}\pi_\tau(\xi))\\
&=\sigma(\chi(\sum_{\xi\in I} \frac{t_\xi}{n_\xi}\pi_\tau(\xi)))\\
&=\sigma(\sum_{\xi\in I}\sum_{i=1}^{n_\xi} \frac{t_\xi}{n_\xi} \pi_h(\xi_{ii})).
\end{align*}
Since $\GG$ is coamenable, the counit extends to a character $\varps\colon A_\red\to \CC$ and we have
\[
\varps(\sum_{\xi\in I} \frac{t_\xi}{n_{\xi}}(\sum_{i=1}^{n_\xi}\xi_{ii}))=\sum_{\xi\in I} \frac{t_\xi}{n_\xi}n_\xi=1.
\]
Hence $1\in \sigma \Big{(}\sum_{\xi\in I} \frac{t_\xi}{n_{\xi}}(\sum_{i=1}^{n_\xi}\pi_h(\xi_{ii}))\Big{)}=\sigma(\lambda_{2,\mu})$ and we conclude that $R(\GG)$ is amenable.

	Assume, conversely, that $R(\GG)$ is amenable. We aim at proving that $\GG$ fulfills the Kesten condition from Theorem \ref{kesten-prop}. Let therefore $u\in \MM_n(A)$ be an arbitrary, finite dimensional, unitary corepresentation. Denote by $(u_{\alpha})_{\alpha\in S}\subseteq \Irred(\GG)$ the irreducible corepresentations occurring in the decomposition of $u$ and by $k_{\alpha}$ the multiplicity of $u_\alpha$ in $u$. Now define
\[
\mu_u(u_\alpha)=
\left\{%
\begin{array}{ll}
   \frac{k_\alpha n_\alpha}{n} & \hbox{if $\alpha\in S$;} \\
    0 & \hbox{if $\alpha\notin S$.} \\
\end{array}%
\right.
\]
Putting $\mu=\frac12 \mu_{u}+\frac12 \mu_{\bar{u}}$ we obtain a finitely supported, symmetric probability measure and by assumption we have that $1\in \sigma(\lambda_{2,\mu})$.  Using again that the character map extends to an injective $*$-homomorphism $\chi\colon C^*_\red(R(\GG))\to A_\red$ we obtain
\begin{align*}
\sigma(\lambda_{2,\mu}) &= \sigma\Big{(}\sum_{\alpha\in S}\frac{k_\alpha n_\alpha}{2n}\lambda_{2, u_\alpha} +\sum_{\alpha\in S}\frac{k_\alpha n_\alpha}{2n}\lambda_{2, u_{\bar{\alpha}}} \Big{)}\\
&=\sigma \Big{(}\sum_{\alpha\in S}\frac{k_\alpha n_\alpha}{2n} l_{u_\alpha} +\sum_{\alpha\in S}\frac{k_\alpha n_\alpha}{2n}l_{u_{\bar{\alpha}}}\Big{)}\tag{Prop. \ref{gns-and-lambda}}\\
&=\sigma\Big{(}\sum_{\alpha\in S}\frac{k_\alpha n_\alpha}{2n}\frac{1}{n_\alpha}\pi_\tau(u_\alpha)+\sum_{\alpha\in S}\frac{k_\alpha n_\alpha}{2n}\frac{1}{n_\alpha}\pi_\tau(u_{\bar{\alpha}})   \Big{)}\tag{Rem. \ref{bounded-rem}}\\
&=\sigma\Big{(}\sum_{\alpha\in S}\frac{k_\alpha}{2n}\pi_h(\chi(u_\alpha))+\sum_{\alpha\in S}\frac{k_\alpha}{2n}\pi_h(\chi(u_{\bar{\alpha}}))\Big{)}   \\
&=\sigma\Big{(} \frac{1}{2n}\pi_h(\chi(u)) +\frac{1}{2n}\pi_h(\chi(\bar{u})) \Big{)}\\
&=\sigma\Big{(} \frac{1}{n}\pi_h(\Re(\chi(u))) \Big{)}.\\
\end{align*}
Thus
\[
1\in \sigma(\lambda_{2,\mu}) \quad\text{ if and only if }\quad n\in\sigma(\Re(\pi_h(\chi(u)))),
\]
and the result now follows from Theorem \ref{kesten-prop}.
\end{proof}
In particular we (re-)obtain the following.
\begin{cor}\label{group-fusion}
A  discrete group is amenable if and only if the group ring, considered as a fusion algebra, is amenable.
\end{cor}

\begin{cor}[\cite{banica-subfactor}]
The quantum groups $SU_q(2)$ are coamenable.
\end{cor}
\begin{proof}
By Theorem \ref{quantum-fusion-link}, $SU_q(2)$ is coamenable if and only if $R(SU_q(2))$ is amenable. But, $R(SU_q(2))=R(SU(2))$ (see e.g.~\cite{woronowicz-tannaka}) and since $(C(SU(2)),\Delta_c)$ is a coamenable quantum group $R(SU(2))$ is amenable. 
\end{proof}
As seen from Theorem \ref{quantum-fusion-link}, the answer to the question of whether a compact quantum group is coamenable or not can be determined using only information about its corepresentations  --- a fact noted by Banica in the setting of compact matrix quantum groups in \cite{banica-fusion} and \cite{banica-subfactor}. With this in mind, we now propose the following F{\o}lner condition for quantum groups.
\begin{defi}\label{quantum-foelner}
A compact quantum group $\GG=(A,\Delta)$ is said to sa\-tis\-fy F{\o}lner's condition if for any finite, non-empty subset $S\subseteq \Irred(\GG)$ and any $\varps>0$ there exists a finite subset $F\subseteq \Irred(\GG)$ such that
\[
\sum_{u\in \del_S(F)}n_u^2< \varps \sum_{u\in F}n_u^2.
\]
Here $n_u$ denotes the dimension of the irreducible corepresentation $u$ and  $\del_S(F)$ is the boundary of $F$ relative to $S$ as in Definition \ref{rand-defi}.
\end{defi}
We immediately obtain the following.

\begin{cor}
A compact quantum group is coamenable if and only if it satisfies F{\o}lner's condition.
\end{cor}
\begin{proof}
By Theorem \ref{quantum-fusion-link}, the compact quantum group $\GG$ is coamenable if and only if $R(\GG)$ is amenable. By Theorem \ref{izumi-thm}, $R(\GG)$ is amenable if and only if it satisfies $(\FC 3)$ which is exactly the same as saying that $\GG$ satisfies F{\o}lner's condition.
\end{proof}

In Section \ref{vanishing-section} we will use this F{\o}lner condition to deduce a vanishing result concerning $L^2$-Betti numbers of compact, coamenable quantum groups.

\section{An Interlude}\label{interlude}
In this section we gather various notation and minor results which will be used in the following section to prove our main result, Theorem \ref{dim-flad}. Some generalities on von Neumann algebraic quantum groups are stated without proofs; we refer to \cite{kustermans-vaes} for the details.

	Consider again a compact quantum group $\GG=(A,\Delta)$ with tracial Haar state $h$. Denote by $\{u^\alpha\mid \alpha\in I\}$ a complete set of representatives for the equivalence classes of irreducible, unitary corepresentations of $\GG$. Consider the dense Hopf $*$-algebra 
\[
A_0=\spann_\CC\{u_{ij}^\alpha\mid \alpha\in I\},
\]
and its discrete dual  Hopf $*$-algebra $\hat{A_0}$. Since $h$ is tracial, the discrete quantum group $\hat{A_0}$ is unimodular; i.e.~the left and right invariant functionals are the same. Denote by $\hat{\varphi}$ the left and right invariant functional on $\hat{A_0}$ normalized such $\hat{\varphi}(h)=1$. For $a\in A_0$ we denote by $\hat{a}\in A_0'$ the map
\[
A_0\ni x {\longmapsto} h(ax)\in \CC.
\] 
Then, by definition, we have  $\hat{A_0}=\{\hat{a}\mid a\in A_0\}$. The algebra $\hat{A_0}$ is $*$-isomorphic to 
\[
\displaystyle\overset{\alg}{\bigoplus_{\alpha\in I}}\ \MM_{n_\alpha}(\CC),
\]
and because $h$ is tracial the isomorphism has a simple description; if we denote by $E_{ij}^\alpha$ the standard matrix units in $\MM_{n_\alpha}(\CC)$ then the map
\[
\Phi (\widehat{(u_{ij}^{\alpha})^*})=\tfrac{1}{n_\alpha }E_{ij}^\alpha,
\]
extends to a $*$-isomorphism \cite{vandaele}.  Denote by $\lambda$ the GNS representation of $A$ on $H=L^2(A_0,h)$, by $\eta$ the canonical inclusion $A_0\subseteq H$ and by $M$ (or $\lambda(M)$) the enveloping von Neumann algebra $\lambda(A_0)''$.  The map $\hat{\eta}\colon\hat{A_0}\to H$ given by $\hat{a}\mapsto \eta(a)$ makes $(H,\hat{\eta})$ a GNS pair for $(\hat{A_0},\hat{\varphi})$ 
%WELL, IN GENERAL ITS A GNS PAIR FOR THE RIGHT-INVARIANT FUNCTIONAL $\hat{psi}$ BUT BECAUSE OF UNIMODULARITY THE TWO FUNCTIONALS ARE THE SAME
and the corresponding GNS representation $L$ is given by
\[
L(\hat{a})\eta(x)=\hat{\eta}(\hat{a}\hat{x}).
\]
We denote by $\hat{M}$ (or $L(\hat{M}))$ the enveloping von Neumann algebra $L(\hat{A_0})''$. This is a discrete von Neumann algebraic quantum group and $\hat{\varphi}$ gives rise to a left and right invariant, normal, semifinite, faithful weight on $\hat{M}$. Each finite subset $E\subseteq I$ gives rise to a central projection 
\[
P_E=\Phi^{-1}\Big{(}\sum_{\alpha\in I}\chi_{E}(\alpha)1_{n_\alpha}\Big{)}\in \hat{A}_0, 
\]
where $1_{n_\alpha}$ denotes the unit in $\MM_{n_{\alpha}}(\CC)$ and $\chi_E$ is the characteristic function for the set $E$. A direct computation shows that $L(P_{E})$ is the orthogonal projection onto the finite dimensional subspace
\[
\spann_\CC\{u_{ij}^{\bar{\alpha}}\mid 1\leq i,j\leq n_\alpha, {\alpha}\in {E}\}.
\]
Recall from Example \ref{corep-fusion} that $u^{\bar{\beta}}$ is the element in $\{u^\alpha\mid \alpha\in I\}$ which is equivalent to $(u^{\beta})^c$. Because $h$ is tracial, the left invariant weight $\hat{\varphi}$ on $\hat{A_0}$ has the particular simple form \cite[p.47]{vaes-et-al-notes}
\[
\hat{\varphi}=\Big{(}\sum_{\alpha\in I}n_{\alpha}\tr_{n_{\alpha}}\Big{)}\circ\Phi,
\]
where $\tr_{n_\alpha}$ is the non-normalized trace on $\MM_{n_{\alpha}}(\CC)$. In particular
\[
\hat{\varphi}(P_E)=\sum_{\alpha\in E}n_{\alpha}^2=\hat{\varphi}(P_{\bar{E}}),
\]
for any finite subset $E\subseteq I$.  For any $m\in M$ and any finite subset $E\subseteq I$ we have \cite[2.10]{vaes-van-daele-heisenberg} that
\begin{align*}
\Tr_{H}(m^*P_E m)=h(m^*m)\hat{\varphi}(P_E),
\end{align*}
where $\Tr_H$ denotes the standard trace on $B(H)$. Here, and in what follows, we suppress the representations $\lambda$ and $L$ of $M$ and $\hat{M}$ respectively on $H$. The commutant $M'$ is the underlying von Neumann algebra of a compact, von Neumann algebraic quantum group whose Haar state is also given by the vector state $h$ and whose discrete dual is given by $(\hat{M},\hat{\Delta})^\op$; this quantum group has $\hat{M}$ as its underlying von Neumann algebra, but is endowed with comultiplication $\hat{\Delta}^\op=\sigma \hat{\Delta}$ where $\sigma$ denotes the flip-automorphism on $\hat{M}\htens \hat{M}$. Since $(\hat{M},\hat{\Delta})$ is unimodular we see that $\hat{\varphi}^\op=\hat{\varphi}$ and hence the trace-formula above extends in the following way.
\begin{lem}[\cite{vaes-van-daele-heisenberg}]\label{trace-formula}
For any $m\in M$ or $m\in M'$ and any finite subset $E\subseteq I$ we have $\Tr_H(m^*P_Em)= h(m^*m)\hat{\varphi}(P_E)$.
\end{lem}
With this lemma we conclude the interlude and move towards an application of the quantum F{\o}lner condition.

\section{A Vanishing Result}\label{vanishing-section}
In this section we investigate the $L^2$-Betti numbers of coamenable quantum groups. The notion of $L^2$-Betti numbers for compact quantum groups was introduced in \cite{quantum-betti} and we refer to that paper (and the introduction) for the definitions and basic results.  Throughout the section, we will freely use L{\"u}ck's extended Murray-von Neumann dimension, but whenever explicit properties are used there will be a reference. These references will be to the original work \cite{luck97} and \cite{luck98}, but for the reader who wants to learn the subject Lück's book \cite{luck02} is probably a better general reference. 

	Consider again a compact quantum group $\GG=(A,\Delta)$ with Haar state $h$ and denote by $M$ the enveloping von Neumann algebra in the GNS representation arising from $h$. As promised in the introduction, we will now prove the following theorem which should be considered as a quantum group analogue of Theorem 5.1 from \cite{luck98}. 

\begin{thm}\label{dim-flad}
If $\GG$ is coamenable and $h$ is tracial then for any left $A_0$-module $Z$ and any $k\geq 1$ we have
\[
\dim_M \Tor_k^{A_0}(M,Z)=0,
\]
where $\dim_M(-)$ is L{\"u}ck's extended dimension function arising from the extension of the trace-state $h$. 
\end{thm}
If $M$ were flat as a module over $A_0$ we would have $\Tor_k^{A_0}(M,Z)=0$ for any $Z$ and any $k\geq 1$, and the property in Theorem \ref{dim-flad} is therefore referred to as \emph{dimension flatness} of the von Neumann algebra over the algebra of matrix coefficients. The proof of Theorem \ref{dim-flad}, which is a generalization of the corresponding proof of \cite[5.1]{luck98}, is divided into three parts. Part I consists of reductions while part II contains the central argument carried out in detail in a special case. Part III shows how to boost the argument from part II to the general case. Throughout the proof, we will use freely the quantum group notation developed in the previous sections without further reference; in particular, $\{u^\alpha\mid \alpha \in I\}$ will denote a fixed, complete set of pairwise inequivalent, irreducible, unitary corepresentations of $\GG$.

%\newpage
\begin{proof}[Proof of Theorem \ref{dim-flad}.]
\hspace{1cm}\\
\vspace{-0.5cm}
\begin{center}
\textsc{Part I}
\end{center}
We begin with some reductions. Let an arbitrary $A_0$-module $Z$ be given and choose a free module $F$ that surjects onto $Z$. Then we have a short exact sequence
\[
0\To K\To F\To Z\To 0,
\]
and since $F$ is free (in particular flat) the corresponding long exact $\Tor$-sequence gives an isomorphism
\[
\Tor_{k+1}^{A_0}(M,Z)\simeq \Tor_{k}^{A_0}(M,K) \ \textrm{ for } k\geq 1.
\]
It is therefore sufficient to prove the theorem for arbitrary $Z$ and $k=1$.  Moreover, we may assume that $Z$ is finitely generated since $\Tor$ commutes with direct limits, every module is the directed union of its finitely generated submodules and $\dim_M(-)$ is well behaved with respect to direct limits \cite[2.9]{luck98}. Actually, we can assume that $Z$ is finitely presented since any finitely generated module $Z$ is a direct limit of finitely presented modules. To see this, choose a short exact sequence
\[
0\To K\To F \To Z \To 0,
\]
with $F$ finitely generated and free. Denote by $(K_j)_{j\in J}$ the directed system of finitely generated submodules in $K$. Then $F/K_j$ is finitely presented for each $j\in J$ and 
\[
Z= \varinjlim_j F/K_j.
\]
Because of this and the direct limit formula for the dimension function \cite[2.9]{luck98} we may, and will, therefore assume that $Z$ is finitely presented. Choose a finite presentation 
\[
A_0^{n}\overset{f}{\To } A_0^m \To Z\To 0.
\]
Put $H=L^2(A,h)$,  $K=\ker(f)\subseteq A_0^n\subseteq H^n$ and denote by $f^{(2)}\colon H^n\to H^m$ the continuous  extension of $f$. Then we have
\[
\Tor_1^{A_0}(M,Z)=\frac{\ker(\id_M\tens f)}{M\utens{A_0} K},
\]
and hence
\begin{align*}
\dim_M\Tor_1^{A_0}(M,Z)&=\dim_M \ker(\id_M\tens f)-\dim_M M\utens{A_0} K\\
&=\dim_M \ker(f^{(2)})-\dim_M \overline{K}^{\|\cdot\|_2},
\end{align*}
where the second equality follows from \cite[2.11]{CS}. See also \cite[p.158-159]{luck98}. So we need to prove that $\overline{K}^{\|\cdot\|_2}=\ker(f^{(2)})$.

\vspace{0.3cm}
\begin{center}
\textsc{Part II}
%\textbf{Part II}
\end{center}
We first treat the case $m=n=1$. Then the map $f$ has the form $R_a$ (right-multiplication by $a$) for some $a\in A_0$. 
If $a=0$ we have $\overline{K}^{\|\cdot\|_2}=H=\ker(f^{(2)})$ so we may assume $a\neq 0$. Since the ${u_{ij}^\alpha}$'s constitute a linear basis for $A_0$, the element $a\in A_0$ has a unique expansion
\begin{align}
a=\sum_{\alpha\in I}\sum_{i,j=1}^{n_\alpha}t_{ij}^\alpha u_{ij}^\alpha,\tag{$t_{ij}^\alpha\in \CC$}
\end{align} 
and we may therefore consider the non-empty, finite set $S\subseteq I$ given by
\[
S=\{\alpha\in I\mid \exists \ 1\leq i,j\leq n_\alpha : t_{ij}^\alpha\neq 0\}.
\]
Denote by $H_0$ the kernel of $f^{(2)}$ and by $q_0\in M'$ the projection onto it. Denote by $q$ the projection onto $H_0\cap K^\perp$; we need to prove that this subspace is trivial and since the vector-state $h$ is faithful on $M'$ this is equivalent to proving $h(q)=0$. Let $\varps>0$ be given. Since $\GG$ is assumed coamenable, the F{\o}lner condition provides the existence of a finite, non-empty subset $F\subseteq I$ such that
\[
\sum_{\alpha\in \del_S(F)} n_\alpha^2< \varps \sum_{\alpha\in F}n_\alpha^2.
\]
Here we identify a subset $E\subseteq I$ with the corresponding set of corepresentations $\{u^\alpha\mid \alpha\in E\}$.
To simplify notation further we will write $\del$ instead of $\del_S(F)$ in the following and moreover we will suppress the GNS-representations $\lambda\colon M\to B(H)$ and $L\colon \hat{M}\to B(H)$ as in Section \ref{interlude}. Since $h$ is tracial, Woronowicz's quantum Peter-Weyl Theorem \cite[3.2.3]{tuset} takes a particular simple form and states that the set 
\[
\{\sqrt{n_\alpha}u_{ij}^\alpha\mid 1\leq i,j\leq n_\alpha,\alpha\in I\}
\]
constitutes an orthonormal basis for $H$. Hence every $x\in H$ has an $\ell^2$-expansion
\begin{align*}
x=\sum_{\alpha\in I}\sum_{i,j=1}^{n_\alpha} x_{ij}^\alpha\sqrt{n_\alpha} u_{ij}^\alpha. \tag{$x_{ij}^\alpha\in \CC$} 
\end{align*}
Consider a vector $x\in H$ and assume that $P_{\bar{\del}}(x)=0$ such that the $\ell^2$-expansion of $x$ has the form $\sum_{\alpha\notin \del}\sum_{i,j=1}^{n_\alpha} x_{ij}^\alpha \sqrt{n_\alpha} u_{ij}^\alpha$. For  $\gamma \in S$ and $1\leq p,q\leq n_\gamma$ we then have 
\begin{align*}
R^{(2)}_{u_{pq}^\gamma} P_{\bar{F}}(x) &=\sum_{\alpha\notin \del,\alpha\in F}\sum_{i,j=1}^{n_\alpha}x_{ij}^\alpha \sqrt{n_\alpha} u_{ij}^\alpha u_{pq}^\gamma\\
P_{\bar{F}} R^{(2)}_{u_{pq}^\gamma}(x) &= P_{\bar{F}}\Big{(}\sum_{\alpha\notin \del}\sum_{i,j=1}^{n_\alpha} x_{ij}^\alpha \sqrt{n_\alpha} u_{ij}^\alpha u_{pq}^\gamma\Big{)}
\end{align*}
Here $R^{(2)}_{u_{pq}^\gamma}$ denotes the $L^2$-extension of $R_{u_{pq}^\gamma}$. Since $u_{ij}^\alpha u_{pq}^\gamma$ is contained in the linear span of the matrix coefficients of $u^\alpha\tenrep u^\gamma$ and since $\alpha\notin \del=\del_S(F)$ and $\gamma\in S$ we see that the two expressions above are equal. By linearity and continuity we obtain
\[
f^{(2)} P_{\bar{F}}(x)=P_{\bar{F}} f^{(2)}(x). 
\]
This holds for all  $x\in \ker(P_{\bar{\del}})$, so if $x\in H_0\cap \ker(P_{\bar{\del}})$ we have 
\[
0=f^{(2)} P_{\bar{F}}(x)=f(P_{\bar{F}}(x)), 
\]
where the last equality is due to the fact that $\rg(P_{\bar{F}})\subseteq A_0\subseteq H$. This proves that $P_{\bar{F}}(x)\in K=\ker(f)$ and since $q$ was defined as the projection onto $H_0\cap K^\perp$ we get $q P_{\bar{F}}(x)=0$. Since this holds whenever $x\in H_0=q_0(H)$ and $P_{\bar{\del}}(x)=0$ we get 
\[
q P_{\bar{F}}(q_0\wedge(1-P_{\bar{\del}}))=0.
\]
Thus, the restriction $q \pfb \colon H_0\to H$ factorizes through $H_0/H_0\cap\ker(\pdb)$ and we have
\begin{align*}
\dim_\CC(q \pfb(H_0))& \leq \dim_\CC(H_0/H_0\cap\ker(\pdb))\\
&\leq \dim_{\CC}(H/\ker(\pdb))\\
&=\dim_\CC(\rg(\pdb))\\
& =\sum_{\alpha\in {\del}}n_\alpha^2\\
&=\hat{\varphi}(P_\del).
\end{align*}
For any finite rank operator $T\in B(H)$ one has 
\[
|\Tr_H(T)|\leq \|T\|\dim_\CC(T(H))
\]
and using this and Lemma \ref{trace-formula} we now get
\begin{align*}
h(q)\hat{\varphi}(P_F)&=h(q)\hat{\varphi}(\pfb)\\
&=\Tr_H(q \pfb q)\\
&\leq \|q \pfb q\|\dim_\CC(q \pfb q (H))\\
&\leq \dim_\CC (q \pfb (H_0)) \\
&\leq \hat{\varphi}(P_\del).
\end{align*}
Thus
\[
h(q)\leq \frac{\hat{\varphi}(P_\del)}{\hat{\varphi}(P_F)}< \varps,
\]
and since $\varps>0$ was arbitrary we conclude that $q=0$.

\vspace{0.5cm}
\begin{center}
\textsc{Part III}
%\textbf{Part III}
\end{center}
We now treat the general case of a finitely presented $A_0$-module $Z$ with finite presentation 
\[
A_0^n \overset{f}{\To} A_0^m \To Z\To 0. 
\]
In this case $f$ is given by right multiplication by an $n\times m$ matrix $T=(t_{ij})$ with entries in $A_0$. Each $t_{ij}$ has a unique linear expansion as $t_{ij}=\sum_{\alpha,k,l}t^{(i,j)}_{\alpha, k,l}u_{kl}^\alpha$ and we put
\[
S=\{\alpha\in I \mid \exists \ i,j,k,l,\alpha\ : t^{(i,j)}_{\alpha,k,l}\neq 0 \}.
\]
As in Part II, we may assume that $T\neq 0$ so that $S\neq \emptyset$.  Denote by $H_0$ the space $\ker(f^{(2)})\subseteq H^n$, by $q_0\in \MM_n(M')$ the projection onto $H_0$ and by $q\in \MM_n(M')$ the projection onto $H_0\cap K^\perp$. We need to show that $q=0$. Denote by $\Tr_n$ the non-normalized trace on $\MM_n(\CC)$ and put $h_n=h\tens \Tr_n \colon B(H)\tens \MM_n(\CC)\to \CC$. We aim at proving that $h_n(q)=0$, which suffices since $h$ is faithful on $M'$. For each $x\in \hat{M}$ we denote by $x^n$ the diagonal operator on $H^n$ which has $x$ in each diagonal entry. Under the identification $B(H)\tens \MM_n(\CC)=B(H^n)$ we see that $\Tr_H\tens \Tr_n$ corresponds to $\Tr_{H^n}$, and Lemma \ref{trace-formula} together with a direct computation therefore gives 
\begin{align}
\Tr_{H^n}(A^*P_E^n A)= h_n(A^*A)\hat{\varphi}(P_E)\tag{$\dagger$},
\end{align}
for any finite subset $E\subseteq I$ and any $A$ in $\MM_n(M)$ or $\MM_n(M')$. Let $\varps>0$ be given and choose according to the F{\o}lner condition a finite subset $F\subseteq I$ such that
\[
\sum_{\alpha\in\del_S(F)}n_\alpha^2<\frac{\varps}{n}\sum_{\alpha\in F}n_\alpha^2,
\]
and put $\del=\del_S(F)$ for simplicity. By repeating the argument from the beginning of Part II we arrive at the equation
\[
q P_{\bar{F}}^n(q_0\wedge(1-P_{\bar{\del}}^n))=0,
\]
which in turn yields
\[
\dim_\CC(q\pfb^n(H_0))\leq \dim_\CC(\rg(\pdb^n))=n\dim_\CC(\rg(\pdb))=n\hat{\varphi}(P_{\del}).
\]
Using the trace-formula $(\dagger)$ we conclude that
\begin{align*}
h_n(q)\hat{\varphi}(P_F) &= \Tr_{H^n}(qP_{\bar{F}}^n q)\\
&\leq\|q P_{\bar{F}}^n q\|\dim_\CC(q P_{\bar{F}}^n q(H))\\
&\leq \dim_\CC(q\pfb^n(H_0))\\
&\leq n\hat{\varphi}(P_{\del}).
\end{align*}
Thus
\[
h_n(q)\leq n\frac{\hat{\varphi}(P_{\del})}{\hat{\varphi}(P_F)}<\varps,
\]
and since $\varps>0$ was arbitrary we conclude that $h_n(q)=0$ as desired.

\end{proof}

By putting $Z= \CC$ in Theorem \ref{dim-flad}, we immediately obtain the following corollary.
\begin{cor}\label{vanishing-cor}
Let $\GG=(A,\Delta)$ be a compact, coamenable quantum group with tracial Haar state. Then
$\bet_n(\GG)=0$ for all $n\geq 1$. Here $\bet_n(\GG)$ is the $n$-th $L^2$-Betti number of $\GG$ as defined in \cite{quantum-betti}.
\end{cor}
In particular we obtain the following  extension of \cite[3.3]{quantum-betti}.
\begin{cor}
For an abelian, compact quantum group $\GG$ we have $\bet_n(\GG)=0$ for $n\geq 1$.
\end{cor}
\begin{proof}
Since $\GG$ is abelian it is of the form $(C(G),\Delta_c)$ for some compact (second countable) group $G$. Since the counit, given by evaluation at the identity, is already globally defined and bounded it is clear that $\GG$ is coamenable and the result now follows from Corollary \ref{vanishing-cor}.
\end{proof}

We also obtain the classical result of L{\"u}ck. 
\begin{cor}{\cite[5.1]{luck98}}\label{luck-thm}
If $\Gamma$ is an amenable, countable, discrete group then for all  $\CC\Gamma$-modules $Z$  and all $n\geq 1$ we have
\[
\dim_{\L(\Gamma)}\Tor_n^{\CC\Gamma}(\L(\Gamma),Z)=0.
\]
In particular, $\bet_n(\Gamma)=0$ for $n\geq 1$.
\end{cor}
\begin{proof}
Put $\GG=(C^*_\red(\Gamma),\Delta_\red)$. Then $\GG$ is coamenable if and only if $\Gamma$ is amenable and the result now follows from Theorem \ref{dim-flad} and Corollary \ref{vanishing-cor}
\end{proof}
Note, however, that this does not really give a new proof of L{\"u}ck's result since the proof of Theorem \ref{dim-flad} coincides with L{\"u}ck's proof of the statement in Corollary \ref{luck-thm} when $\GG=(C^*_\red(\Gamma),\Delta_\red)$. \\

	In \cite{CS}, Connes and Shlyakhtenko introduced a notion of $L^2$-Betti numbers for tracial $*$-algebras. From the above results we also obtain vanishing of these Connes-Shlyakhtenko $L^2$-Betti numbers for certain Hopf $*$-algebras. More precisely we get the following.

\begin{cor}
Let $\GG=(A,\Delta)$ be a compact, coamenable quantum group with tracial Haar state $h$. Then $\bet_n(A_0,h)=0$ for all $n\geq 1$, where $\bet_n(A_0,h)$ is the $n$-th  Connes-Shlyakhtenko $L^2$-Betti number of the $*$-algebra $A_0$ with respect to the trace $h$.
\end{cor}

\begin{proof}
By \cite[4.1]{quantum-betti} we have $\bet_n(\GG)=\bet_n(A_0,h)$ and the claim therefore follows from Corollary \ref{vanishing-cor}.
\end{proof}
The knowledge of dimension flatness also gives genuine homological information about the ring extension $A_0\subseteq M$. More precisely, the following holds.

\begin{cor}\label{eksakt}
If $\GG=(A,\Delta)$ is compact and coamenable with tracial Haar state then the induction functor $M\odot_{A_0}-$ is an exact functor from the category of finitely generated, projective $A_0$-modules to the category of finitely generated, projective $M$-modules.
\end{cor}

\begin{proof}
Let $X$ and $Y$ be finitely generated, projective $A_0$-modules and let $f\colon X \to Y$ be an injective homomorphism.  Then
\[
0\To X\overset{f}{\To} Y \To Y/\rg(f)\To 0,
\]
is a projective resolution of $Y/\rg(f)$. Thus $\Tor_1^{A_0}(M,Y/\rg(f))=\ker(\id_M\tens f)$ and from Theorem \ref{dim-flad} we conclude that 
\[
\dim_M(\ker(\id_M\tens f))=0. 
\]
Because $\id_M\tens f$ is a map of finitely generated projective $M$-modules, it is not difficult to prove that 
\[
\ker(\id_M\tens f)=\overline{\ker(\id_M\tens f)}^\alg,
\]
where $\overline{\ker(\id_M\tens f)}^\alg$ is defined (see \cite{luck98}) as the intersection of all kernels arising from homomorphisms from $M\odot_{A_0}X$ to $M$ vanishing on $\ker(\id_M\tens f)$. By \cite[0.6]{luck98}, we conclude from this that $\ker(\id_M\tens f)$ is finitely generated and projective. But, since the dimension function is faithful on the category of finitely generated, projective $M$-modules this forces $\ker(\id_M\tens f)=\{0\}$ and the claim follows.
\end{proof}
Corollary \ref{eksakt}, in particular, implies the following result which was pointed out to us by A.~Thom.
\begin{cor}
Let $\GG=(A,\Delta)$ be compact and coamenable with tracial Haar state and let $x\in A_0$ be a non-zero element such that there exists a non-zero  $m\in M$ with $mx=0$. Then there exists a non-zero $y\in A_0$ with $yx=0$.
\end{cor}
\begin{proof}
This follows by using Corollary \ref{eksakt} on the  map $a\longmapsto ax$.
\end{proof}
An analogous statement about products in the opposite order follows by using the involution in $M$. So, formulated in ring theoretical terms, we obtain the following: Any regular element in $A_0$ stays regular in the over-ring $M$.

\section{Examples}
A concrete  example of a non-commutative, non-cocommutative, coamenable  (matrix) quantum group with tracial Haar state is the orthogonal quantum group $A_{o}(2)\simeq SU_{-1}(2)$. It follows from \cite[5.1]{banica-fusion} that $A_o(2)$ is coamenable. To see that the Haar state is tracial, one observes that the orthogonality property of the canonical fundamental corepresentation implies that the antipode has period two. 
\subsection{Examples arising from tensor products}
If $\GG_1=(A_1,\Delta_1)$ and $\GG_2=(A_2,\Delta_2)$ are compact quantum groups then the (minimal) tensor product $A=A_1\tens A_2$ may be turned into a quantum group $\GG$ by defining the comultiplication $\Delta\colon A\To A\tens A$ to be
\[
\Delta(a)=(\id\tens\sigma\tens \id)(\Delta_1\tens\Delta_2)(a),
\]
where $\sigma$ denotes the flip-isomorphism from $A_1\tens A_2$ to $A_2\tens A_1$. The Haar state is the tensor product of the two Haar states and the counit is the tensor product of the counits. Using these facts, it is not difficult to see \cite{murphy-tuset} that if both $\GG_1$ and $\GG_2$ are coamenable and have tracial Haar states, then the same is true for $\GG$. See e.g.~\cite[11.3.2]{KR2}.

\subsection{Examples arising from bicrossed products}\label{bicrossed-section}

Another way to obtain examples of compact, coamenable quantum groups is via \emph{bicrossed products}.  We therefore briefly sketch the bicrossed product construction following \cite{vaes-vainerman} closely. In \cite{vaes-vainerman}, Vaes and Vainerman consider the more general notion of \emph{cocycle} bicrossed products, but since we will mainly be interested in the case where the cocycles are trivial we will restrict our attention to this case in the following. The more general situation will be discussed briefly in Remark \ref{cocycle-rem}. The bicrossed product construction is defined using the language of von Neumann algebraic quantum groups. We will use this language freely in the following and refer to \cite{kustermans-vaes} for the background material.

	Let $(M_1,\Delta_1)$ and $(M_2,\Delta_2)$ be locally compact (l.c.) von Neumann algebraic quantum groups. Let $\tau\colon M_1\htens M_2\to M_1\htens M_2$ be a faithful $*$-homomorphism and denote by $\sigma\colon M_1\htens M_2\to M_2\htens M_1$ the flip-isomorphism. Then $\tau$ is called a \emph{matching} from $M_1$ to $M_2$ if the following holds.
\begin{itemize}
\item The map $\alpha\colon M_2\To M_1\htens M_2$ given by $\alpha(y)=\tau(1\tens y)$ is a (left) coaction of $(M_1,\Delta_1)$ on the von Neumann algebra $M_2$. 
\item Defining $\beta\colon M_1\To M_1\htens M_2$ as $\beta(x)=\tau(x\tens 1)$  the map $\sigma \beta$ is a (left) coaction of $(M_2,\Delta_2)$ on the von Neumann algebra $M_1$.
\item The coactions satisfy  the following two \emph{matching conditions}:
\begin{align*}
\tau_{(13)}(\alpha\tens 1)\Delta_2  &=(1\tens \Delta_2)\alpha\tag{M1}\\
\tau_{(23)}\sigma_{(23)}(\beta\tens 1)\Delta_1 &= (\Delta_1\tens 1)\beta\tag{M2}
\end{align*}

%HER KAN MAN INDFï¿½RE DE ALTERNATIVE MATCHINGBETINGELSER. DE ER EN ANELSE Pï¿½NERE, MEN IKKE SPECIELT RELEVANTER HER
\end{itemize}
Here we use the standard leg numbering convention (see e.g.~\cite{vandaele}). If $\tau\colon M_1\htens M_2\to M_1\htens M_2$ is a matching from $M_1$ to $M_2$ then it is easy to see that $\sigma\tau\sigma^{-1}$ is a matching from $M_2$ to $M_1$. We will therefore just refer to the pair $(M_1,M_2)$ as a matched pair and to $\tau$ as a matching of the pair. Let $(M_1,\Delta_1)$ and $(M_2,\Delta_2)$ be such a matched pair of l.c.~quantum groups and denote by $\tau$ the matching. We denote by $H_i$ the GNS space of $M_i$ with respect to the left invariant weight $\varphi_i$ and by $W_i$ and $\hat{W_i}$ the natural multiplicative unitaries on $H_i\htens H_i$ for $M_i$ and $\hat{M_i}$  respectively. By $H$ we denote $H_1\htens H_2$ and by $\Sigma$ the flip-unitary on $H\htens H$.
%Then $\tilde{\tau}=\sigma\tau\sigma$ matches the pair $(M_2,M_1)$ and gives rise to coactions $\tilde{\alpha}$ and $\tilde{\beta}$ as above. 
We may now form two crossed products:
\begin{align*}
M&=M_1\ltimes_\alpha M_2  =\vna\{\alpha(M_2),\hat{M_1}\tens 1\} \subseteq \B(H_1\htens H_2)\\
\tilde{M}&=M_2\ltimes_{\sigma\beta} M_1 =\vna\{\sigma\beta(M_1),\hat{M_2}\tens 1\}\subseteq \B(H_2\htens H_1)
\end{align*}
Some of the main results in \cite{vaes-vainerman} are summarized in the following:
\begin{thm}[\cite{vaes-vainerman}]
Define operators
\[
\hat{W}=(\beta\tens 1\tens 1)(W_1\tens 1)(1\tens1\tens\alpha)(1\tens\hat{W_2})
\]
and $W=\Sigma \hat{W}^* \Sigma$ on $H\htens H$. Then $W$ and $\hat{W}$ are multiplicative unitaries and the map $\Delta\colon M\to \B(H\htens H)$ given by $\Delta(a)=W^*(1\tens 1 \tens a)W$ defines a comultiplication on $M$ turning it into a l.c.~quantum group. 
Denoting by $\Sigma_{12}$ the flip-unitary from $H_1\htens H_2$ to $H_2\htens H_1$, the dual quantum group $\hat{M}$ becomes $\Sigma_{12}^* \tilde{M}\Sigma_{12}$ with comultiplication implemented by $\hat{W}$.
\end{thm}
Thus, up to a flip the two crossed products above are in duality. In \cite{vaes-bicrossed-amenability}, Desmedt, Quaegebeur and Vaes studied (co)amenability of bicrossed products. Combining their Theorem 15 with \cite[2.17]{vaes-vainerman} we obtain the following: If $(M_1,M_2)$ is a matched pair with $M_1$ discrete and $M_2$ compact then the bicrossed product $M$ is compact, and $M$ is coamenable if and only if both $M_2$ and $\hat{M}_1$ are. Here a von Neumann algebraic compact quantum group is said to be coamenable if the corresponding $C^*$-algebraic quantum group is. Collecting the results discussed above we obtain the following.
\begin{prop}\label{bicrossed-vanishing}
If $(M_1,M_2)$ is a matched pair of l.c.~quantum groups in which $\hat{M}_1$ and $M_2$ are compact and coamenable, then the bicrossed product $M=M_1\ltimes_\alpha M_2$ is coamenable and compact.  So if the Haar state on $M$ is tracial the quantum group $(M,\Delta)$ has vanishing $L^2$-Betti numbers in all positive degrees. 
\end{prop}

%%%%%%%%%%%%%%%%%%%%%%%%%%%%%%%%%%%%%%%%%%%%%%%%%%%%%%%%%%%%%%%%%%%%%%%%%%%%%%%%%%%%%%%%%%%%%%%%%%%
In order to produce more concrete examples, we will now discuss a special case of the bicrossed product construction in which one of the coactions comes from an actual group action. This part of the theory is due to De Canni{\`e}re \cite{canniere} and is formulated using the language of Kac algebras. We remind the reader that a \emph{compact} Kac algebra is nothing but a von Neumann algebraic, compact quantum group with tracial Haar state. A discrete, countable group $\Gamma$ acts on a compact Kac algebra $(M,\Delta, S,h)$ if the group acts on the von Neumann algebra $M$ and the action commutes with both the coproduct and the antipode. Denoting the action by $\rho$, this means that
\begin{align*}
\Delta(\rho_\gamma(x))&=\rho_\gamma\tens\rho_\gamma (\Delta(x)),\\
S(\rho_\gamma(x)) &=\rho_\gamma(S(x)),%
\end{align*}
for all $\gamma\in \Gamma$ and all $x\in M$. 
In this situation, the action of $\Gamma$ on $M$ induces a coaction $\alpha \colon M\To \ell^\infty(\Gamma)\htens M$. Denoting by $H$ the Hilbert space on which $M$ acts and
identifying $\ell^2(\Gamma)\htens H$ with $\ell^2(\Gamma,H)$, this coaction is given by the formula 
\[
\alpha(x)(\xi)(\gamma)=\rho_{\gamma^{-1}}(x)(\xi(\gamma)),
\]
for $\xi\in\ell^2(\Gamma,H)$. The crossed product, which is defined as 
\[
\Gamma\ltimes_\rho M=\{\alpha(M),\L(\Gamma)\tens 1\}'', 
\]
becomes again a Kac algebra \cite[Thm.1]{canniere}. One should note at this point that De Canni{\`e}re works with the \emph{right} crossed product acting on $H\htens \ell^2(\Gamma)$ where we work with the \emph{left} crossed product acting on $\ell^2(\Gamma)\htens H$. But, one can come from one to the other by conjugation with the flip-unitary and we may therefore freely transport all results from \cite{canniere} to the setting of \emph{left} crossed products. We now prove that De Canni{\`e}re's crossed product can also be considered as a bicrossed product. This is probably well known to experts in the field, but we were unable to find an explicit reference.
\begin{prop}
Defining $\tau\colon \ell^\infty(\Gamma)\htens M\To \ell^{\infty}(\Gamma)\htens M$ by 
\[
\tau(\delta_\gamma\tens x)=\delta_\gamma\tens \rho_{\gamma^{-1}}(x)
\]
we obtain a matching with the above defined $\alpha$ as the corresponding coaction of $\ell^\infty(\Gamma)$ on $M$ and trivial coaction of $(M,\Delta)$ on $\ell^\infty(\Gamma)$. 
\end{prop}
\begin{proof}
A direct calculation shows that $\alpha(x)=\tau(1\tens x)$ and $\beta(f)=\tau(f\tens 1)=f\tens 1$. Therefore the two maps $x\mapsto \tau(1\tens x)$ and $f\mapsto \sigma \tau(f\tens 1)$ are coactions as required. We therefore just have to check that the matching conditions are fulfilled. Denote the coproduct on $\ell^\infty(\Gamma)$ by $\Delta_1$ and choose $f\in \ell^\infty(\Gamma)$ such that $\Delta_1(f)\in \ell^\infty(\Gamma)\odot\ell^\infty(\Gamma)$. Writing $\Delta_1(f)$ as $f_{(1)}\tens f_{(2)}$ we now get
\begin{align*}
\tau_{(23)}\sigma_{(23)}(\beta\tens 1)\Delta_1 f &= \tau_{(23)}\sigma_{(23)}(\beta\tens 1)(f_{(1)}\tens f_{(2)})\\
&=\tau_{(23)}\sigma_{(23)}(f_{(1)}\tens 1\tens f_{(2)})\\
&=\tau_{(23)}(f_{(1)}\tens f_{(2)}\tens 1)\\
&=f_{(1)}\tens f_{(2)}\tens 1\\
&=(\Delta_1\tens 1)\beta(f),
\end{align*} 
and hence (M2) is satisfied. An analogous, but slightly more cumbersome, calculation proves that (M1) is also satisfied.
%DET MAN Gï¿½R ER, AT BENYTTE DET ALTERNATIVE MATCHINGBETINGELSER FRA VAES-VAINERMAN SIDE 18. DEN ANDEN VAR LET AT TJEKKE %I VORES SETUP OG DEN Fï¿½RSTE ER GANSKE LET AT TJEKKE Nï¿½R MAN BENYTTER DEN ANDEN FORMULERING
\end{proof}

Thus, as von Neumann algebras, we have $\ell^\infty(\Gamma)\ltimes_\alpha M= \Gamma \ltimes_\rho M$. Using the fact that $\beta$ is trivial, one can prove that the elements $\lambda_\gamma\tens 1$ are group-like and it therefore follows from \cite[3.3]{canniere} that also the comultiplications agree. Hence the two crossed product constructions are identical as l.c.~quantum groups. In particular, the the bicrossed product $\ell^\infty(\Gamma)\ltimes_\alpha M$ is a Kac algebra so if $(M,\Delta)$ is compact then $\ell^\infty(\Gamma)\ltimes_\alpha M$ is also compact \cite[2.7]{vaes-vainerman} and the Haar state is tracial. We therefore have the following.
\begin{prop}
If $\GG=(M,\Delta,S,h)$ is a compact, coamenable Kac algebra and $\Gamma$ is a countable, discrete, amenable group acting on $\GG$ then the crossed product $\Gamma\ltimes M$ is again a compact, coamenable Kac algebra. 
\end{prop}
\begin{proof}
That $\Gamma\ltimes M$ is a Kac algebra follows from the discussion above and the coamenability of the crossed product follows from \cite[15]{vaes-bicrossed-amenability} since $\widehat{\ell^{\infty}(\Gamma)}=\L(\Gamma)$ is coamenable if (and only if) $\Gamma$ is amenable.
\end{proof}

\begin{rem}\label{cocycle-rem}
It is also possibly to construct examples using the more general notion of cocycle crossed products introduced in \cite[2.1]{vaes-vainerman}. It is shown in \cite[13]{vaes-bicrossed-amenability} that weak amenability (i.e.~the existence of an invariant mean) is preserved under cocycle bicrossed products. In general it is not known whether or not weak amenability is equivalent to strong amenability, the latter being defined as the dual quantum group being coamenable in the sense of Definition \ref{coamenable-defi}. But for discrete quantum groups this equivalence has been proven by Tomatsu in \cite{tomatsu-amenable} and also by Blanchard and Vaes in unpublished work. Therefore, if $(M_1,M_2)$ is a cocycle matched pair of l.c.~quantum groups with both $\hat{M}_1$ and $M_2$ compact and coamenable, then the cocycle crossed product is also compact and coamenable.
\end{rem}

%%%%%%%%%%%%%%%%%%%%%%%%%%%%%%%%%%%%%%%%%%%%%%%%%%%%%%%%BIBLIOGRAPHY%%%%%%%%%%%%%%%%%%%%%%%%%%%%%%%%%%%%%%%%%%
%\newpage
\newcommand{\etalchar}[1]{$^{#1}$}


\begin{thebibliography}{MVD98}

\bibitem[Ati76]{atiyah}
M.~F. Atiyah.
\newblock Elliptic operators, discrete groups and von {N}eumann algebras.
\newblock In {\em Colloque ``Analyse et Topologie'' en l'Honneur de Henri
  Cartan (Orsay, 1974)}, pages 43--72. Ast\'erisque, No. 32--33. Soc. Math.
  France, Paris, 1976.

\bibitem[Ban99a]{banica-fusion}
Teodor Banica.
\newblock Fusion rules for representations of compact quantum groups.
\newblock {\em Exposition. Math.}, 17(4):313--337, 1999.

\bibitem[Ban99b]{banica-subfactor}
Teodor Banica.
\newblock Representations of compact quantum groups and subfactors.
\newblock {\em J. Reine Angew. Math.}, 509:167--198, 1999.

\bibitem[BMT01]{murphy-tuset}
E.~B{\'e}dos, G.~J. Murphy, and L.~Tuset.
\newblock Co-amenability of compact quantum groups.
\newblock {\em J. Geom. Phys.}, 40(2):130--153, 2001.

\bibitem[BP92]{benedetti}
Riccardo Benedetti and Carlo Petronio.
\newblock {\em Lectures on hyperbolic geometry}.
\newblock Universitext. Springer-Verlag, Berlin, 1992.

\bibitem[BS93]{baaj-skandalis}
Saad Baaj and Georges Skandalis.
\newblock Unitaires multiplicatifs et dualit\'e pour les produits crois\'es de
  {$C\sp *$}-alg\`ebres.
\newblock {\em Ann. Sci. \'Ecole Norm. Sup. (4)}, 26(4):425--488, 1993.

\bibitem[CS05]{CS}
Alain Connes and Dimitri Shlyakhtenko.
\newblock {$L\sp 2$}-homology for von {N}eumann algebras.
\newblock {\em J. Reine Angew. Math.}, 586:125--168, 2005.

\bibitem[DC79]{canniere}
Jean De~Canni{\`e}re.
\newblock Produit crois\'e d'une alg\`ebre de {K}ac par un groupe localement
  compact.
\newblock {\em Bull. Soc. Math. France}, 107(4):337--372, 1979.

\bibitem[DQV02]{vaes-bicrossed-amenability}
Pieter Desmedt, Johan Quaegebeur, and Stefaan Vaes.
\newblock Amenability and the bicrossed product construction.
\newblock {\em Illinois J. Math.}, 46(4):1259--1277, 2002.

\bibitem[ES92]{enock-schwartz}
Michel Enock and Jean-Marie Schwartz.
\newblock {\em Kac algebras and duality of locally compact groups}.
\newblock Springer-Verlag, Berlin, 1992.
\newblock With a preface by Alain Connes, With a postface by Adrian Ocneanu.

\bibitem[F{\o}l55]{foelner}
Erling F{\o}lner.
\newblock On groups with full {B}anach mean value.
\newblock {\em Math. Scand.}, 3:243--254, 1955.

\bibitem[Ger88]{gerl}
Peter Gerl.
\newblock Random walks on graphs with a strong isoperimetric property.
\newblock {\em J. Theoret. Probab.}, 1(2):171--187, 1988.

\bibitem[HI98]{izumi}
Fumio Hiai and Masaki Izumi.
\newblock Amenability and strong amenability for fusion algebras with
  applications to subfactor theory.
\newblock {\em Internat. J. Math.}, 9(6):669--722, 1998.

\bibitem[HY00]{yamagami-hayashi}
Tomohiro Hayashi and Shigeru Yamagami.
\newblock Amenable tensor categories and their realizations as {AFD} bimodules.
\newblock {\em J. Funct. Anal.}, 172(1):19--75, 2000.

\bibitem[Kes59]{kesten}
Harry Kesten.
\newblock Full {B}anach mean values on countable groups.
\newblock {\em Math. Scand.}, 7:146--156, 1959.

\bibitem[KR83]{KR1}
Richard~V. Kadison and John~R. Ringrose.
\newblock {\em Fundamentals of the theory of operator algebras. {V}ol. {I}},
  volume 100 of {\em Pure and Applied Mathematics}.
\newblock Academic Press Inc. [Harcourt Brace Jovanovich Publishers], New York,
  1983.
\newblock Elementary theory.

\bibitem[KR86]{KR2}
Richard~V. Kadison and John~R. Ringrose.
\newblock {\em Fundamentals of the theory of operator algebras. {V}ol. {II}},
  volume 100 of {\em Pure and Applied Mathematics}.
\newblock Academic Press Inc., Orlando, FL, 1986.
\newblock Advanced theory.

\bibitem[KT99]{tuset}
Johan Kustermans and Lars Tuset.
\newblock A survey of {$C\sp *$}-algebraic quantum groups. {I}.
\newblock {\em Irish Math. Soc. Bull.}, 43:8--63, 1999.

\bibitem[KV03]{kustermans-vaes}
Johan Kustermans and Stefaan Vaes.
\newblock Locally compact quantum groups in the von {N}eumann algebraic
  setting.
\newblock {\em Math. Scand.}, 92(1):68--92, 2003.

\bibitem[Kye08]{quantum-betti}
David Kyed.
\newblock {$L^2$}-homology for compact quantum groups.
\newblock {\em Math.Scand.}, 103(1):111--129, 2008.

\bibitem[L{\"u}c97]{luck97}
Wolfgang L{\"u}ck.
\newblock Hilbert modules and modules over finite von {N}eumann algebras and
  applications to {$L\sp 2$}-invariants.
\newblock {\em Math. Ann.}, 309(2):247--285, 1997.

\bibitem[L{\"u}c98a]{luck98}
Wolfgang L{\"u}ck.
\newblock Dimension theory of arbitrary modules over finite von {N}eumann
  algebras and {$L\sp 2$}-{B}etti numbers. {I}. {F}oundations.
\newblock {\em J. Reine Angew. Math.}, 495:135--162, 1998.

\bibitem[L{\"u}c98b]{luckto}
Wolfgang L{\"u}ck.
\newblock Dimension theory of arbitrary modules over finite von {N}eumann
  algebras and {$L\sp 2$}-{B}etti numbers. {II}. {A}pplications to
  {G}rothendieck groups, {$L\sp 2$}-{E}uler characteristics and {B}urnside
  groups.
\newblock {\em J. Reine Angew. Math.}, 496:213--236, 1998.

\bibitem[L{\"u}c02]{luck02}
Wolfgang L{\"u}ck.
\newblock {\em {$L\sp 2$}-invariants: theory and applications to geometry and
  {$K$}-theory}, volume~44 of {\em Ergebnisse der Mathematik und ihrer
  Grenzgebiete. 3. Folge. A Series of Modern Surveys in Mathematics [Results in
  Mathematics and Related Areas. 3rd Series. A Series of Modern Surveys in
  Mathematics]}.
\newblock Springer-Verlag, Berlin, 2002.

\bibitem[MVD98]{vandaele}
Ann Maes and Alfons Van~Daele.
\newblock Notes on compact quantum groups.
\newblock {\em Nieuw Arch. Wisk. (4)}, 16(1-2):73--112, 1998.

\bibitem[Rua96]{ruan-amenability}
Zhong-Jin Ruan.
\newblock Amenability of {H}opf von {N}eumann algebras and {K}ac algebras.
\newblock {\em J. Funct. Anal.}, 139(2):466--499, 1996.

\bibitem[Sun92]{sunder-bimodules}
V.~S. Sunder.
\newblock {${\rm II}\sb 1$} factors, their bimodules and hypergroups.
\newblock {\em Trans. Amer. Math. Soc.}, 330(1):227--256, 1992.

\bibitem[Tom06]{tomatsu-amenable}
Reiji Tomatsu.
\newblock Amenable discrete quantum groups.
\newblock {\em J. Math. Soc. Japan}, 58(4):949--964, 2006.

\bibitem[VKV{\etalchar{+}}]{vaes-et-al-notes}
Stefaan Vaes, Johan Kustermans, Leonid Vainerman, Alfons van Daele, and
  Stanislaw Woronowicz.
\newblock Locally compact quantum groups.
\newblock Lecture Notes, \url{www.wis.kuleuven.be/analyse/stefaan}.

\bibitem[Voi79]{voiculescu-amenability}
Dan Voiculescu.
\newblock Amenability and {K}atz algebras.
\newblock In {\em Alg\`ebres d'op\'erateurs et leurs applications en physique
  math\'ematique (Proc. Colloq., Marseille, 1977)}, volume 274 of {\em Colloq.
  Internat. CNRS}, pages 451--457. CNRS, Paris, 1979.

\bibitem[VV03]{vaes-vainerman}
Stefaan Vaes and Leonid Vainerman.
\newblock Extensions of locally compact quantum groups and the bicrossed
  product construction.
\newblock {\em Adv. Math.}, 175(1):1--101, 2003.

\bibitem[VVD03]{vaes-van-daele-heisenberg}
Stefaan Vaes and Alfons Van~Daele.
\newblock The {H}eisenberg commutation relations, commuting squares and the
  {H}aar measure on locally compact quantum groups.
\newblock In {\em Operator algebras and mathematical physics ({C}onstan\c ta,
  2001)}, pages 379--400. Theta, Bucharest, 2003.

\bibitem[Wor87]{woronowicz-pseudo}
S.~L. Woronowicz.
\newblock Compact matrix pseudogroups.
\newblock {\em Comm. Math. Phys.}, 111(4):613--665, 1987.

\bibitem[Wor88]{woronowicz-tannaka}
S.~L. Woronowicz.
\newblock Tannaka-{K}re\u\i n duality for compact matrix pseudogroups.
  {T}wisted {${\rm SU}(N)$} groups.
\newblock {\em Invent. Math.}, 93(1):35--76, 1988.

\bibitem[Wor98]{wor-cp-qgrps}
S.~L. Woronowicz.
\newblock Compact quantum groups.
\newblock In {\em Sym\'etries quantiques (Les Houches, 1995)}, pages 845--884.
  North-Holland, Amsterdam, 1998.

\bibitem[Yam99]{yamagami}
Shigeru Yamagami.
\newblock Notes on amenability of commutative fusion algebras.
\newblock {\em Positivity}, 3(4):377--388, 1999.

\end{thebibliography}
\end{document}